\pdfoutput=1
\RequirePackage{snapshot}
\documentclass[10pt]{iopart}

\usepackage[utf8]{inputenc}
\usepackage[T1]{fontenc}

\usepackage[pdftex]{hyperref}
\usepackage{setstack}
\usepackage{tikz}
\usepackage{pgfplots}
\pgfplotsset{compat=1.7}
\usepgfplotslibrary{fillbetween}
\usepackage[export]{adjustbox}
\usepackage{subfig} 
\usepackage{lmodern}
\usepackage{float}

\newlength{\myx} 
\newlength{\myy}

\expandafter\let\csname equation*\endcsname\relax 
\expandafter\let\csname endequation*\endcsname\relax 
\expandafter\let\csname leftroot\endcsname\relax 
\expandafter\let\csname uproot\endcsname\relax 
\expandafter\let\csname boxed\endcsname\relax
\expandafter\let\csname dddot\endcsname\relax
\expandafter\let\csname ddddot\endcsname\relax
\expandafter\let\csname overset\endcsname\relax 
\expandafter\let\csname underset\endcsname\relax 
\expandafter\let\csname sideset\endcsname\relax 
\expandafter\let\csname subarray\endcsname\relax 
\expandafter\let\csname endsubarray\endcsname\relax
\expandafter\let\csname substack\endcsname\relax
\usepackage{amsmath,amssymb,epsfig}

\usepackage{xcolor}

\usepackage{amsthm}

\newtheorem{theorem}{Theorem}[section]

\newtheorem{lemma}[theorem]{Lemma}

\theoremstyle{definition}

\newtheorem{remark}{Remark}

\newtheorem{example}{Example}[section]

\mathchardef\ordinarycolon\mathcode`\:
\mathcode`\:=\string"8000\def\R{\mathbb{R}}
\begingroup \catcode`\:=\active
  \gdef:{\mathrel{\mathop\ordinarycolon}}
\endgroup

\makeatletter
\newsavebox\myboxA
\newsavebox\myboxB
\newlength\mylenA

\newcommand*\xoverline[2][0.75]{%
    \sbox{\myboxA}{$\m@th#2$}%
    \setbox\myboxB\null
    \ht\myboxB=\ht\myboxA%
    \dp\myboxB=\dp\myboxA%
    \wd\myboxB=#1\wd\myboxA
    \sbox\myboxB{$\m@th\overline{\copy\myboxB}$}
    \setlength\mylenA{\the\wd\myboxA}
    \addtolength\mylenA{-\the\wd\myboxB}%
    \ifdim\wd\myboxB<\wd\myboxA%
       \rlap{\hskip 0.5\mylenA\usebox\myboxB}{\usebox\myboxA}%
    \else
        \hskip -0.5\mylenA\rlap{\usebox\myboxA}{\hskip 0.5\mylenA\usebox\myboxB}%
    \fi}
\makeatother

\renewcommand{\R}{\mathbb{R}}

\usepackage{mathrsfs}

\renewcommand{\d}{\,\mathrm{d}}						  

\DeclareMathOperator*{\supp}{\mathrm{supp}}           
\DeclareMathOperator*{\diag}{\mathrm{diag}}

\renewcommand{\triangle}{\vartriangle}
\renewcommand{\epsilon}{\varepsilon}
\renewcommand{\rho}{\varrho}

\newcommand{\be}{\begin{equation}}
\newcommand{\ee}{\end{equation}}
\newcommand{\bea}{\begin{eqnarray}}
\newcommand{\eea}{\end{eqnarray}}
\newcommand{\bean}{\begin{eqnarray*}}
\newcommand{\eean}{\end{eqnarray*}}

\newcommand{\bel}[1]{\begin{equation}\label{#1}}

\newcommand{\eel}[1]{{\label{#1}\end{equation}}}

\newcommand{\vect}[1]{\boldsymbol{#1}}

\newcommand{\Nn}{\mathbb{N}}
\newcommand{\Hh}{\mathcal{H}}
\newcommand{\Pp}{\mathcal{P}}
\newcommand{\Rr}{\mathbb{R}}

\newcommand{\vph}{\varphi}

\newcommand{\rrr}{\vect{r}}

\newcommand{\Mm}{\xoverline[0.6]{\mathrm{m}}}
\newcommand{\Mmm}{\xoverline[0.6]{\vect{M}}}

\usepackage[babel]{csquotes}
\usepackage{array,booktabs, tabularx}
\usepackage{longtable}
\usepackage[nolist,nohyperlinks]{acronym} 				

\usepackage{thmtools}

\begin{document}

\title[3D MPI model using realistic magnetic field topologies ]{A new 3D model for magnetic particle imaging using realistic magnetic field topologies for algebraic reconstruction}

\author{ Ga\"el Bringout$^{1,2}$,  Wolfgang Erb$^3$, J\"urgen Frikel$^4$}
\address{
$^1$ Universit{\"a}t zu L\"ubeck, Ratzeburger Allee 160, 23562 L{\"u}beck, Germany \\
$^2$ Physikalisch-Technische Bundesanstalt (PTB), Abbestr. 2-12,
10587 Berlin-Charlottenburg, Germany \\
 $^3$ Universit{\`a} degli Studi di Padova, Via Trieste 63, 35121 Padova, Italy \\
 $^4$ Ostbayerische Technische Hochschule Regensburg, Galgenbergstr. 32, 93053 Regensburg, Germany \\
}
\ead{gael.bringout@gmail.com, wolfgang.erb@lissajous.it, juergen.frikel@oth-regensburg.de}

\begin{abstract}
We derive a new 3D model for \ac{MPI} that is able to incorporate realistic magnetic fields in the reconstruction process. In real \ac{MPI} scanners, the generated magnetic fields have distortions that lead to deformed magnetic low-field volumes (LFV) with the shapes of ellipsoids or bananas instead of ideal field-free points (FFP) or lines (FFL), respectively. Most of the common model-based reconstruction schemes in MPI use however the idealized assumption of an ideal FFP or FFL topology and, thus, generate artifacts in the reconstruction. Our model-based approach is able to deal with these distortions and can generally be applied to dynamic magnetic fields that are approximately parallel to their velocity field. We show how this new 3D model can be discretized and inverted algebraically in order to recover the magnetic particle concentration. To model and describe the magnetic fields, we use decompositions of the fields in spherical harmonics. We complement the description of the new model with several simulations and experiments. 
\end{abstract}

\vspace{-2mm}


\vspace{-2mm}



\section{Introduction}
The smart design of magnetic coils for the generation of oscillating magnetic fields is a key challenge in Magnetic Particle Imaging (MPI) \cite{bringout2016MPI,KnoppBuzug2012}. The generated magnetic fields combined with the non-linear magnetization response of the tracer material consisting of \ac{SPIONS} determine the signal acquisition process in MPI and, ultimately, how the distribution of \ac{SPIONS} can be reconstructed. For this reason an accurate description and analysis of realistic magnetic fields is essential to study modelling and reconstruction in MPI.

Since the first publication in 2005 \cite{GleichWeizenecker2005}, Magnetic Particle Imaging has undergone major development steps based on a few major designs for the generation of magnetic fields. For two of these topologies, generally referred to as field-free point (FFP) and field-free line (FFL) topology, fully 3D commercial preclinical MPI scanners are available at the present moment that are able to track \ac{SPIONS} with a high sensitivity and a high temporal resolution. This makes the biomedical imaging modality MPI to a promising tracer-based diagnostic tool, in particulur for blood flow imaging or for quantitative stem cell imaging \cite{Knopp2017,Panagiotopoulos2015,saritas2013magnetic}. 

In the original scanner design for MPI developed at Philips research (introduced in \cite{GleichWeizenecker2005}, extended to a full in vivo 3D design in \cite{Weizenecker2009}), a static gradient field with a space-homogeneous time-varying drive field is combined in order to magnetize the \ac{SPIONS}. The two fields are generated in such a way that a moving spot is created in which the resulting magnetic field is low. The center of this low-field spot in which the magnetic field ideally vanishes is called the \ac{FFP}. As soon as the \ac{FFP} moves over a distribution of \ac{SPIONS}, the magnetization of the \ac{SPIONS} starts to flip, inducing a measurable voltage signal in one or several receive coils. From this time-dependent voltage signal the position of the \ac{SPIONS} can be reconstructed. In the original Philips design, the \ac{FFP} of the created field moves along a 3D Lissajous trajectory inside a rectangular volume. Later on, also other \ac{FFP}-trajectories have been introduced in MPI, see \cite{Knoppetal2009}. 

In \cite{Weizenecker2008}, a second major design principle for magnetic fields was introduced in which the applied magnetic fields ideally vanish along a \ac{FFL}. Compared to the \ac{FFP} setting, the voltage signal is now created in a much larger low field region along the field-free line, providing a higher sensitivity \cite{Weizenecker2008} during the scan. 
A second main advantage of the \ac{FFL} topology is the availability of an efficient model-based reconstruction formula based on the inverse Radon transform \cite{Knopp_etal2011hm}.

For both field topologies, \ac{FFP} and \ac{FFL}, models for the reconstruction of the particle density have been derived. However, only in very idealized settings, as for a 1D-\ac{FFP} along line segments \cite{Erbetal2018,GoodwillConolly2010,Rahemeretal2009} or a non-rotating FFL \cite{bringout2016MPI,Erbe2014,Knopp_etal2011hm}, simple and reliable reconstruction formulas are available. While these simple formulas can be incorporated successfully also in 2D and 3D reconstructions \cite{GoodwillConolly2011, marz2016model,Schomberg2010}, they lead to artifacts once the directions of the magnetic fields are altering quickly, as for instance if the \ac{FFP} is moving on a 3D Lissajous curve. This discrepancy is due to limited possibilities to describe the magnetization behavior of nonuniform anisotropic \ac{SPIONS} correctly if the external magnetic fields are changing rapidly their orientation. In this case, complex
numerical simulations of the Fokker-Planck equations for coupled Brown/N{\'e}el rotations are necessary to describe the imaging properly, see \cite{Kluth2018,Kluth2019,Weizenecker2018}.

Moreover, in practice, real magnetic fields involved in the generation of a \ac{FFP} 
or a \ac{FFL} contain distortions. In particular in the \ac{FFL} setting, the \ac{LFV} of the field has more the appearance of a slightly bended banana than that of a stright line \cite{bringout2016MPI,Erbe2014}. While using the inverse Radon transform for signals created by an ideal \ac{FFL} yields a reasonable recovery of the particle concentration, this is no longer the case for realistic magnetic fields. In this case, the given distortions lead to serious artifacts in the reconstruction, in particular at the boundary of the field of view (as illustrated in the Tables~\ref{tab:reco_normal_sino} and~\ref{tab:reco_0padded_sino} below). A further problem arises from the particular dynamic generation of the \ac{FFL}. In order to accelerate the signal acquisition process, the \ac{FFL} is continuously rotated with a frequency $f_{\mathrm{rot}}$ \cite{KnoppErbe_etal2010}. As the classical filtered backprojection is computed on a rectangular grid in Radon space, the regridding from the rotated Radon information causes additional artifacts in the reconstruction. 

The goal of this article is to introduce and study a new 3D model for MPI that is able to incorporate realistic magnetic fields, and to provide a simple
reconstruction algorithm at the same time. More precisely, for realistic uni-directional time-oscillating magnetic fields we aim at obtaining a model-based reconstruction formula that generalizes the known \ac{FFL} and 1D-\ac{FFP} formulas in MPI. This new model-based reconstruction reduces artifacts generated by distortions and the rotation dynamics of the magnetic fields and allows us to calculate the particle concentration in an efficient way based on an algebraic method. 

To this end, we introduce a family of magnetic fields in which the field is parallel to its own velocity field. For this family of fields, the direction of the field does not change over time, allowing to substitute the general MPI imaging equation with a simpler 3D integral equation that can be discretized in an efficient way.  In the mathematical formulation we use spherical harmonic expansions of magnetic fields (as introduced in \cite{bringout2016MPI,Bringout2014}) and, as important examples, we show that this modelling framework includes classical (ideal) models, like the 1D-\ac{FFP} along a straight line \cite{Erbetal2018,Rahemeretal2009} and the ideal \ac{FFL} \cite{Knopp_etal2011hm}. In particular, we show that this formulation offers enough flexibility to model realistic magnetic fields, e.g. in a \ac{FFL}-type setting, by including higher order harmonics into the expansion. The coefficients of the higher order harmonics can be measured in a calibration procedure providing a realistic MPI model for a particular scanner. In this context, our MPI modelling framework can be interpreted as a hybrid between model-based and measurement-based approach in which the parameters of the magnetic fields are determined in a preliminary step.

\subsection{Contributions} \vspace{-1mm}

\begin{itemize}
 \item[(i)] We introduce a new modelling framework in MPI based on the expansion of magnetic fields in spherical harmonics and homogeneous harmonic polynomials, and we show how ideal and realistic magnetic field topologies in MPI can be modeled within this framework.
 \item[(ii)] We state a new $3D$ MPI model for magnetic fields in which the velocity and the acceleration field are parallel. Applied to ideal cases, this general model explains the standard 1D-FFP and FFL reconstruction formulas.  
 \item[(iii)] We use this new model to obtain a model-based reconstruction that is able to handle realistic magnetic fields in FFL-type imaging. This new model-based approach is able to significantly reduce artifacts in the reconstruction caused by idealized assumptions on the magnetic fields. 
 \item[(iv)] We give a numerical implementation of this reconstruction scheme and provide several simulations and experiments complementing our results.
\end{itemize}

\subsection{Outline of the paper} \vspace{-1mm}
We continue this introductury part by giving a brief overview about the general imaging concepts in MPI (Section~\ref{sec:MPIsignalgeneration}). We further give a mathematical description of important ideal and realistic magnetic field topologies encountered in MPI (Section~\ref{sec:idealmagneticfields}). 
The new model used for the algebraic reconstruction of the particle concentration with realistic field topologies is derived in Section~\ref{sec:newMPImodel}. It is formulated in terms of magnetic fields that are parallel to their velocity field. This familiy of fields contains all relevant ideal and realistic topologies in the considered FFL-type imaging scenario. The numerical details to obtain a discrete system matrix from the given continuous model, including approximation and discretization techniques, are provided in Section~\ref{sec:numerics}.
Finally, the experiments in Section~\ref{sec:experiments} show that the new algebraic reconstruction approach based on a model with realistic magnetic fields is very promising and outperforms a direct reconstruction using a filtered back projection. 
We conclude this article in Section~\ref{sec:conclusion}.

\section{Principles of MPI signal generation} \label{sec:MPIsignalgeneration}

\subsection{General imaging model in MPI}

The basic concept of Magnetic Particle Imaging (MPI) is to recover a density $c(\rrr)$ of \ac{SPIONS} from their non-linear magnetization in
an applied time-varying magnetic field $\vect{B}(\rrr,t)$. 
In an MPI scanner, this change in the magnetization of the superparamagnetic particles is measuered in terms of voltage signals induced in one or several receive coils. Neglecting particle-particle interactions, the corresponding general imaging equation is determined by Faraday's law of induction and is given as (\cite[Eq. (2.36)]{KnoppBuzug2012})
\begin{equation}
\label{eq:faradays law}
	u_{\nu}(t) = -\mu_0 {\frac {\mathrm {d} }{\mathrm {d} t}} \int_{\Omega} \left\langle \vect{\rho}_{\nu}(\rrr), \Mmm(\vect{B}(\rrr,t)) \right\rangle \, c(\rrr) \, \mathrm{d}\rrr.
\end{equation}
Here, $u_{\nu}(t)$ denotes the induced voltage in the receive coil $\nu \in \{ 1, \ldots, V\}$ and $\vect{\rho}_{\nu}(\rrr)$ the sensitivity vector of the receive coil $\nu$ pointing in direction of the central axis of the coil. The function $c(\rrr)$ denotes the particle density at the point $\rrr$ in the domain $\Omega \subset \Rr^3$. Finally, $\Mmm$ describes the magnetization response of a single mean SPION depending on the applied magnetic field $\vect{B}(\rrr,t)$.  

This equation describes a general imaging situation in MPI. For a particular measurement setup the sensitivities $\vect{\rho}_{\nu}$, the magnetization response function $\Mmm$ and the employed magnetic fields $\vect{B}(\rrr,t)$ have to be modelled or specified. 

\subsection{Magnetic fields and magnetization}

The magnetization $\Mmm$ of a single SPION is aligned along the direction
of the applied magnetic field $\vect{B}(\rrr,t)$ and can be written as
\begin{equation} \label{eq:Magnetizationmodel1}
\Mmm(\vect{B}(\rrr,t)) = \Mm(|\vect{B}(\rrr,t)|) \ \frac{\vect{B}(\rrr,t)}{|\vect{B}(\rrr,t)|},
\end{equation}
A classical way to describe the modulus $\Mm$ of the magnetization is the Langevin theory of paramagnetism. In this theory the mean modulus $\Mm$ is modeled as
\begin{equation} \label{eq:Langevinmodel2}
\Mm(|\vect{B}|) = m_{0} L( \lambda |\vect{B}|),
\end{equation}
with the Langevin function $L(x)$ and the constant $\lambda$ given by 
\[L(x) =  \coth x - \frac{1}{x} \quad \text{and} \quad \lambda = \frac{\mu_0 m_0}{k_{\mathrm{B}} T }. \]
Here, $m_0$ denotes the magnetic moment of a single SPION, $\mu_0$ the permeability in free space, $k_{\mathrm{B}}$ the Boltzmann constant and
$T$ the temperature. The Langevin function $L(x)$ is a point symmetric function with respect to the origin and converges for $x \to \pm \infty$ to $\lim_{x\to\pm\infty} L(x) = \pm 1$. Its derivative is given by
\begin{equation} \label{eq:Langevinmodelderivative}
	L'(x) = \begin{cases}
	\displaystyle
		\frac{1}{3}, &x=0,\\[1ex]
		\frac{1}{ x^2}- \frac{1}{\sinh^{2}(x)}, & x\neq 0.
	\end{cases}
\end{equation}
For the modulus $\Mm$, we therefore get $\Mm'(|\vect{B}|) = m_0 \lambda L'(\lambda |\vect{B}|)$ and asymptotically $\lim_{|\vect{B}|\to\infty} \Mm(|\vect{B}|)= m_0$. For a large vector field strength $|\vect{B}|$ the saturation of the magnetization is therefore described by $m_0$. The derivative $\Mm'(0) = m_0 \lambda/3$ is a measure for the magnetic susceptibility of a particle. 

\subsection{Extension of curl- and divergence-free magnetic fields}

In a volume with no magnetic field source and constant permeability $\mu_0$, as for instance in the interior 
of a cylindrical coil, a magnetic field $\vect{B}$ can be regarded both as a divergence-free and a curl-free vector field (see \cite[Section 2.1.2]{bringout2016MPI} or \cite[Section 5.4]{jackson1999}), i.e. it satisfies the two equations
\[ \nabla \cdot \vect{B} = 0 \quad \text{and} \quad \nabla \times \vect{B} = \vect{0}. \]
The second identity implies that $\vect{B}$ is locally a conservative vector field and can be written as the gradient $\vect{B} = \nabla \varphi_{\vect{B}}$
of a potential function $\varphi_{\vect{B}}$. The fact that $\vect{B}$ is divergence-free then implies that $\varphi_{\vect{B}}$ satisfies the Laplace equation 
$\Delta \varphi_{\vect{B}} = 0$. In particular, assuming that the vector field $\vect{B}$ is sufficiently smooth, the identity $\Delta \varphi_{\vect{B}} = 0$ implies that also every component $B_j$ of the vector field $\vect{B} = (B_1, B_2, B_3)$ is a solution of the Laplace equation $\Delta B_j = 0$, $j \in \{1,2,3\}$. As proposed in \cite{Bringout2014}, this fact enables us to extend the components $B_j$ of the vector field $\vect{B}$ in terms of spherical harmonics in order to get a compact description of the magnetic fields
in MPI.  

\subsubsection{Homogeneous harmonic polynomials.}
A homogeneous polynomial in $\Rr^3$ of degree $l \in \Nn_0$ is a linear combination of the monomials
\[ x^{i_1} y^{i_2} z^{i_3}, \quad i_1 + i_2 + i_3 = l.\]
The space $\Pp_l$ of all homogeneous polynomials of degree $l$ has the dimension
\[ \dim \Pp_l = \frac{(l+1)(l+2)}{2}.\]
Herein, the subspace $\Hh_l$ of homogeneous harmonic polynomials of degree $l$ is given by the polynomials $p \in \Pp_l$ satisfying the Laplace equation
\[ \Delta p (x,y,z) = 0. \]
In this way, the spaces $\Hh_l$ are natural candidates to approximate and expand the components $B_i$ of the magnetic field $\vect{B}$. The dimension of the harmonic spaces $\Hh_l$ is given by $\dim \Hh_l = \dim \Pp_l - \dim \Pp_{l-2} = 2l+1$.

\subsubsection{Spherical harmonics.}
A homogeneous harmonic polynomial $p \in \Hh_l$ of degree $l \in \Nn_0$ can be written in spherical coordinates $(r,\theta,\vph)$ as a linear combination of particular basis functions defined in terms of spherical harmonics. Namely, with the spherical coordinates given by
\[ ( x ,y, z ) = ( r \sin \theta \cos \vph, r \sin \theta \sin \vph, r \cos \theta ),\]
the polynomials
\[ p_{l,m}(r,\theta,\vph) = r^l Y_{l,m}(\theta,\vph),\quad m = -l, \ldots, l,\]
form a basis for the space $\Hh_l$ of harmonic polynomials. Here, $Y_{l,m}(\theta,\vph)$ denote the $2l+1$ real-valued Schmidt semi-normalized spherical harmonics of degree $l$ given by
\[ Y_{l,m}(\theta,\vph) = \left\{ \begin{array}{ll} \textstyle \sqrt{2\frac{(l-m)!}{(l+m)!}} P_{l}^{m}(\cos \theta) \cos (m \vph) & \text{if $m \in \{1, \ldots, l\}$,} \\ 
P_{l}^{0}(\cos \theta) & \text{if $m = 0$,} \\ \textstyle \sqrt{2\frac{(l-m)!}{(l+m)!}}
 P_{l}^{|m|}(\cos \theta) \sin (|m| \vph) & \text{if $m \in \{-1, -2, \ldots, -l\}$,}\end{array}\right.\]
where $P_{l}^{m}$, $l,m \in \Nn$, $0 \leq m \leq l$, denote the associated Legendre polynomials given by
\[ P_{l}^{m}(x)={\frac {1}{2^{l } l !}}(1-x^{2})^{m/2}\ {\frac {\mathrm{d}^{l +m}}{\mathrm{d}x^{l +m}}}(x^{2}-1)^{l}.\]
Note that we omit the frequently used Condon-Shortly phase $(-1)^m$ in the definition of $P_{l}^{m}$ (see \cite[Section 3.5]{jackson1999}).
In Table \ref{tab:1}, the spherical harmonics and the corresponding harmonic polynomials up to degree $l=2$ are listed. 

\begin{center}
\begin{table} 
\caption{Spherical harmonics and harmonic homogeneous polynomials of degree $l \leq 2$.} \label{tab:1}\centering
\begin{tabular}{lll}\hline
Degree $l$ & $Y_{l,m}(\theta,\vph)$ & $p_{l,m}(\rrr)$ \\ \hline \\[-2mm]

$0$ & $\begin{array}{l} Y_{0,0}(\theta,\vph) = 1 \end{array}$ & $\begin{array}{l} p_{0,0}(\rrr) = 1 \end{array} $ \\[1mm]
$1$ & $\begin{array}{l} Y_{1,1}(\theta,\vph) = \sin\theta \cos \vph\\ Y_{1,0}(\theta,\vph) = \cos \theta  \\ Y_{1,-1}(\theta,\vph) = \sin \theta\sin \vph\end{array}$ & $\begin{array}{l} p_{1,1}(\rrr) = x \\ p_{1,0}(\rrr) = z \\ p_{1,-1}(\rrr) = y \end{array}$\\[6mm]
$2$ & $\begin{array}{l} Y_{2,2}(\theta,\vph) = \frac{\sqrt{3}}{2} \sin^2 \theta \cos 2 \vph\\ Y_{2,1}(\theta,\vph) = \sqrt{3} \sin \theta \cos \theta \cos \vph \\ Y_{2,0}(\theta,\vph) = \frac{1}{2} (3 \cos^2\theta -1) \\ Y_{2,-1}(\theta,\vph) = \sqrt{3} \sin \theta \cos \theta \sin \vph\\ Y_{2,-2}(\theta,\vph) = \frac{\sqrt{3}}{2} \sin^2 \theta \sin 2 \vph \end{array}$ 
& $\begin{array}{l} p_{2,2}(\rrr) = \frac{\sqrt{3}}{2} (x^2 - y^2)\\ p_{2,1}(\rrr) = \sqrt{3} xz \\ p_{2,0}(\rrr) = z^2 - \frac{1}{2} x^2 - \frac{1}{2} y^2 \\ p_{2,-1}(\rrr) = \sqrt{3} yz \\ p_{2,-2}(\rrr) = \sqrt{3} x y \end{array}$
\\ \hline
\end{tabular}
\end{table}
\end{center}

\subsubsection{Expansion of the components $B_j$ in spherical harmonics.} 
If the components $B_j$ of the magnetic field $\vect{B} = (B_1, B_2, B_3)$ satisfy the Laplace equation, we can expand them in terms of homogeneous harmonic polynomials and obtain the decomposition
\[ B_j(x,y,z) = \sum_{l=0}^{\infty} \sum_{m = -l}^l c_{l,m}^{j} p_{l,m}(x,y,z), \quad j \in \{1,2,3\}, \]
or, in spherical coordinates,
\[ B_j(r,\theta,\vph) = \sum_{l=0}^{\infty} \sum_{m = -l}^l c_{l,m}^{j} r^l Y_{l,m}(\theta,\vph), \quad j \in \{1,2,3\}. \]
We assume that the magnetic fields are properly smooth, so that there are no issues with convergence at this place.

\section{Ideal and realistic magnetic fields in MPI} \label{sec:idealmagneticfields}

All major magnetic field topologies in \ac{MPI} can be written compactly in terms of spherical harmonic expansions for a few involved generating fields. In the following, we review the classical (ideal) magnetic field topologies in MPI and provide their representations with respect to spherical harmonics. We also explain how such expansions can be obtained for the realistic magnetic field topologies.

\subsection{Ideal field-free point (FFP) on 3D-Lissajous trajectory.} 

\begin{table}
\caption{Spherical harmonic encoding 
of a 3D-Lissajous FFP} 
\label{tab:2}
\centering
\begin{tabular}{lllll}\hline 
Coil name & $B_1$ & $B_2$ & $B_3$ & Time dependent part \\ \hline \\[-2mm]
Selection & $c_{11}^1 = -g$ & $c_{1-1}^2 = -g$ & $c_{10}^3 = 2g$ & $1$ \\
x-drive & $c_{00}^1 = d_x$ &  &  & $\sin 2\pi f_x t$ \\
y-drive &  & $c_{00}^2 = d_y$ &  & $\sin 2\pi f_y t$ \\
z-drive &  &  & $c_{00}^3 = d_z$ & $\sin 2\pi f_z t$ \\ \hline \\
\end{tabular} 
\end{table}

As a first example we consider the building elements of the magnetic field $\vect{B}(\rrr,t)$ in the original Philips design \cite{Weizenecker2009}. In this field topology a \ac{FFP} along a 3D-Lissajous trajectory inside a cuboid domain is created. The spherical harmonic coefficients of the involved drive and selection fields are summarized in Table \ref{tab:2}. Here, the constant $g$ denotes the gradient strength of the selection field, $d_x$, $d_y$, $d_z$, and
$f_x$, $f_y$, $f_z$ the amplitudes and frequencies of the time-dependent drive field. The entire magnetic field $\vect{B}(\rrr,t)$ to generate the 3D-FFP on the Lissajous curve is then given by
\begin{align}
\vect{B}(\rrr,t) &=  \vect{B}_{\mathrm{Selection}}(\rrr) 
  + \vect{B}_{\mathrm{x-drive}}(\rrr) \; \sin 2 \pi f_x t \notag \\
 & + \vect{B}_{\mathrm{y-drive}}(\rrr) \; \sin 2 \pi f_y t 
  + \vect{B}_{\mathrm{z-drive}}(\rrr) \; \sin 2 \pi f_z t \notag \\
 &= g \begin{pmatrix} -x \\ -y \\ 2z \end{pmatrix} +  \begin{pmatrix} d_x \sin 2 \pi f_x t \\ d_y \sin 2 \pi f_y t \\ d_z \sin 2 \pi f_z t \end{pmatrix}. \label{eq:FFP-Lissajous}
\end{align}
The vector field $\vect{B}(\rrr,t)$ is curl- and divergence free with the potential function
\[ \vph_{\vect{B}}(\rrr,t) = g(z^2- \textstyle \frac{x^2}{2}- \textstyle \frac{y^2}{2}) + d_x x \sin (2 \pi f_x t) 
+ d_y y \sin (2 \pi f_y t) + d_z z \sin (2 \pi f_z t).  \]

The field-free point $\rrr_{\mathrm{FFP}}(t)$ itself is the point in $\Rr^3$ at which the magnetic field $\vect{B}(\rrr,t)$ vanishes, i.e.,  
$\vect{B}(\rrr_{\mathrm{FFP}}(t),t) = \vect{0}$. In this example, we have
\begin{align*}
\rrr_{\mathrm{FFP}}(t)  &= \left( \frac{d_x}{g} \sin 2 \pi f_x t, \frac{d_y}{g} \sin 2 \pi f_y t, -\frac{d_z}{2g} \sin 2 \pi f_z t \right).
\end{align*}
In particular, the FFP moves along a Lissjous trajectory inside the cuboid domain $[-|\frac{d_x}{g}|,|\frac{d_x}{g}|] \times [-|\frac{d_y}{g}|,|\frac{d_y}{g}|]\times [-|\frac{d_z}{2g}|,|\frac{d_z}{2g}|] \subset \Rr^3 $.
For Lissajous FFP topologies, model-based reconstruction approaches in 3D or 2D have limitations due to the complex magnetization behavior of \ac{SPIONS}. The reconstruction of the particle density for Lissajous FFP topologies is therefore usually performed by measuring the system responses in a rather time-consuming calibration procedure \cite{Gruettner_etal2013,Kaethner2016IEEE,Weizenecker2009}.

\subsection{Ideal 1D-FFP along line segments.} \label{sec:FFP1D}
To generate a FFP that moves along a line segment in $\Rr^3$, we can apply the field
\begin{align*}
\vect{B}(\rrr,t)
 &= g \begin{pmatrix} -x \\ -y \\ 2z \end{pmatrix} +  \begin{pmatrix} d_x  \\ d_y \\ d_z \end{pmatrix} \sin 2 \pi f_{\mathrm{d}} t.
\end{align*}
This field can be generated in the same way as the magnetic field \eqref{eq:FFP-Lissajous} for the 3D-FFP by using the same drive-field frequency $f_{\mathrm{d}} = f_{x} = f_{y} = f_z$ in all coordinates. The position of the FFP is then given as
\begin{align*}
\rrr_{\mathrm{FFP}}(t)  &= \vect{v} \sin (2 \pi f_{\mathrm{d}} t), \quad \text{with} \quad \vect{v} = \left(\frac{d_x}{g}, \frac{d_y}{g}, -\frac{d_z}{2g}\right). 
\end{align*}
The FFP is now moving in the 1D line segment 
$ \mathbb{L}_{\vect{v}} = \{ \rrr = s \vect{v} \ | \ s \in [-1,1]\}$. Such a 1D-FFP topology is generally used for 1D-MPI imaging, see \cite{Erbetal2018,GoodwillConolly2010,Rahemeretal2009}.

\subsection{Ideal rotating field-free line (FFL) in the $xy$-plane} \label{sec:rotatingFFL}
\begin{table}
\caption{Spherical harmonic coefficients of a 2D rotating FFL} 
\label{tab:3}
\centering
\begin{tabular}{lllll}\hline
Coil name & $B_1$ & $B_2$ & $B_3$ & Time dependent part \\ \hline \\[-2mm]
Select Maxwell & $c_{11}^1 = -g$ & $c_{1-1}^2 = -g$ & $c_{10}^3 = 2g$ & $1$ \\
Select Quad $0$ & $c_{11}^1 = g$ & $c_{1-1}^2 = -g$ &  & $\cos 2\pi f_{\mathrm{rot}} t$ \\
Select Quad $45$ & $c_{1-1}^1 = g$ & $c_{11}^2 = g$ &  & $\sin 2\pi f_{\mathrm{rot}} t$ \\
x-drive & $c_{00}^1 = d$ &  &  & $\sin 2\pi f_{\mathrm{d}} t \; \sin \pi f_{\mathrm{rot}} t$ \\
y-drive &  & $c_{00}^2 = d$ &  & $-\sin 2\pi f_{\mathrm{d}} t \; \cos \pi f_{\mathrm{rot}} t$ \\ \hline
\end{tabular}
\vspace{-2mm}
\end{table}

A magnetic field topology to generate a dynamically rotating field-free line (FFL) in the $xy$-plane was developed in 
\cite{Erbe2014}. The building elements of this rotating FFL are listed in
Table \ref{tab:3}, see also \cite{bringout2016MPI} for a derivation.

The complete magnetic field $\vect{B}(\rrr,t)$ to generate the rotating FFL is given by
\begin{align}
\vect{B}(\rrr,t) &= \vect{B}_{\mathrm{Maxwell}}(\rrr) + \vect{B}_{\mathrm{Quad0}}(\rrr) \cos (2 \pi f_{\mathrm{rot}} t) 
 + \vect{B}_{\mathrm{Quad45}}(\rrr) \sin (2 \pi f_{\mathrm{rot}} t) \notag \\
 &\quad + \vect{B}_{\mathrm{x-drive}}(\rrr) \sin (2\pi f_{\mathrm{d}} t) \sin (\pi f_{\mathrm{rot}} t)
  - \vect{B}_{\mathrm{y-drive}}(\rrr) \sin (2 \pi f_{\mathrm{d}} t) \cos (\pi f_{\mathrm{rot}} t) \notag \\
  &=  g \begin{pmatrix} -x \\ -y \\ 2z \end{pmatrix} + 
  g \begin{pmatrix} x \\ -y \\ 0 \end{pmatrix} \cos (2\pi f_{\mathrm{rot}} t)    + 
  g \begin{pmatrix} y \\  x \\ 0 \end{pmatrix} \sin (2\pi f_{\mathrm{rot}} t) \notag \\ & \quad + \begin{pmatrix} d \sin (2\pi f_{\mathrm{d}} t) \sin (\pi f_{\mathrm{rot}} t)  \\ -d \sin (2\pi f_{\mathrm{d}} t) \cos (\pi f_{\mathrm{rot}} t)  \\ 0 \end{pmatrix}, \label{eq:FFL-rotation}
\end{align}
where $f_{\mathrm{d}}$ and $f_{\mathrm{rot}}$ denote the drive and rotation frequencies of the FFL. 
The potential function $\vph_{\vect{B}}$ of the conservative vector field $\vect{B}(\rrr,t)$ has the form 
\[ \small \vph_{\vect{B}}(\rrr,t) = g \Big( z^2 - \big( x \sin(\pi f_{\mathrm{rot}} t ) - 
y \cos(\pi f_{\mathrm{rot}} t )\big)^2\Big) + d \big(x \sin(\pi f_{\mathrm{rot}} t ) - 
y \cos(\pi f_{\mathrm{rot}} t ) \big)  \sin(2 \pi f_{\mathrm{d}} t). \]
We denote by $\mathrm{FFL}(t)$ the set of all $\rrr$ at which the field $\vect{B}(\rrr,t)$ vanishes at $t \in \Rr$.

\begin{lemma} \label{lem:1} The field-free line $\mathrm{FFL}(t)$ at a time $t \in \mathbb{R}$ can be parametrized as
\begin{align*}
\mathrm{FFL}(t) &= \left\{ \begin{pmatrix} \cos(\pi f_{\mathrm{rot}} t) \\ \sin(\pi f_{\mathrm{rot}} t) \\ 0 \end{pmatrix} h + \frac{d}{2g} \begin{pmatrix} \sin (\pi f_{\mathrm{rot}} t) \sin (2\pi f_{\mathrm{d}} t) \\ - \cos (\pi f_{\mathrm{rot}} t) \sin (2\pi f_{\mathrm{d}} t) \\ 0 \end{pmatrix} \ | \ h \in \Rr\right\} \\ &=
\left\{ (x,y,0):\; \sin(\pi f_{\mathrm{rot}} t) x - \cos(\pi f_{\mathrm{rot}} t) y = \frac{d}{2g} \sin (2\pi f_{\mathrm{d}} t ) \right\}.
\end{align*}
The set $\mathrm{FFL}(t)$ is a line in the $xy$-plane perpendicular to $(\sin(\pi f_{\mathrm{rot}} t), - \cos (\pi f_{\mathrm{rot}} t),0)$ and with distance $|\frac{d}{2g} \sin (2\pi f_{\mathrm{d}} t )|$ to the origin.
\end{lemma}

\begin{proof}
Every point $\rrr = (x,y,z)$ in the set $\mathrm{FFL}(t)$ satisfies by definition $\vect{B}(\rrr,t) = \vect{0}$. Therefore, the formula \eqref{eq:FFL-rotation} for $\vect{B}(\rrr,t)$ gives $z = 0$, and for $x$ and $y$ the system of equations
\[g \begin{pmatrix} x (1 - \cos (2\pi f_{\mathrm{rot}} t)) \\ y (1 + \cos (2\pi f_{\mathrm{rot}} t))  \end{pmatrix} - 
  g \begin{pmatrix} y \\  x  \end{pmatrix} \sin (2\pi f_{\mathrm{rot}} t) = \begin{pmatrix} d \sin (2\pi f_{\mathrm{d}} t) \sin (\pi f_{\mathrm{rot}} t)  \\ -d \sin (2\pi f_{\mathrm{d}} t) \cos (\pi f_{\mathrm{rot}} t)  \end{pmatrix}. 
\]
Factoring out the term $\sin (\pi f_{\mathrm{rot}} t)$ in the first line and $\cos (\pi f_{\mathrm{rot}} t)$ in the second, we can simplify this expression as
\begin{equation} \label{eq:FFLparallel} \left( \sin(\pi f_{\mathrm{rot}} t) x - \cos(\pi f_{\mathrm{rot}} t) y - \frac{d}{2g} \sin (2\pi f_{\mathrm{d}} t ) 
\right) \begin{pmatrix} \sin (\pi f_{\mathrm{rot}} t)  \\ \cos (\pi f_{\mathrm{rot}} t)  \end{pmatrix} = \vect{0}. 
\end{equation}
Identity \eqref{eq:FFLparallel} implies that $\sin(\pi f_{\mathrm{rot}} t) x - \cos(\pi f_{\mathrm{rot}} t) y - \frac{d}{2g} \sin (2\pi f_{\mathrm{d}} t ) = 0$ and, thus, the second stated characterization of the FFL. In particular, it implies that all admissible points $(x,y)$ lie on a line in the 
$xy$-plane perpendicular to $(\sin(\pi f_{\mathrm{rot}} t), - \cos (\pi f_{\mathrm{rot}} t),0)$. From this normal form of the FFL the parametrized description of the FFL given in Lemma \ref{lem:1} follows with a standard linear algebra argument. 
\end{proof}

\subsection{Non-rotating FFL in the $xy$ plane} \label{sec:nonrotatingFFL}
We can slightly modify the magnetic fields from the last subsection to generate non-rotating FFL's. For this, it is only necessary to substitute the time-depending rotating angle $2 \pi f_{\mathrm{rot}} t $ in \eqref{eq:FFL-rotation} with a fixed angle $\alpha \in [0,2 \pi]$. The corresponding magnetic field for a non-rotating FFL is given by
\begin{align*}
\vect{B}(\rrr,t) =  2g \begin{pmatrix} (  y \cos \frac{\alpha}{2} - x \sin \frac{\alpha}{2}) \sin \frac{\alpha}{2} \\  (  x \sin \frac{\alpha}{2} - y \cos \frac{\alpha}{2}) \cos \frac{\alpha}{2} \\ z \end{pmatrix}  + \begin{pmatrix} d \sin (2\pi f_{\mathrm{d}} t) \sin \frac{\alpha}{2}  \\ -d \sin (2\pi f_{\mathrm{d}} t) \cos \frac{\alpha}{2}  \\ 0 \end{pmatrix}. \label{eq:FFL-norotation}
\end{align*}

\noindent From Lemma \ref{lem:1} we can derive that $\mathrm{FFL}(t)$ is in this case given as
\begin{align*}
\mathrm{FFL}(t) &= \left\{ \rrr \ | \ \langle \vect{e}_{\alpha}, \rrr \rangle = \frac{d}{2g} \sin (2\pi f_{\mathrm{d}} t ) \right\},
\end{align*}
where $\vect{e}_{\alpha} = (\sin \left( \frac{\alpha}{2} \right), - \cos \left( \frac{\alpha}{2} \right),0)$ denotes the normal vector of the FFL in the $xy$-plane. In particular, the direction of the FFL is now independent of the time $t$.

\subsection{Realistic magnetic fields in magnetic particle imaging} \label{sec:realistidfield}

While the ideal field topologies of the last sections are represented by only a few spherical harmonics, a realistic magnetic field is accurately described by a much larger amount of spherical harmonic coefficients. By incorporating these higher order spherical harmonics in the description of the magnetic fields a refined and scanner-adapted model of the magnetic field topologies is obtained. The magnetic field coefficients in the expansion can be determined by measuring the field at a discrete set of spherical nodes with a subsequent numerical evaluation of the spherical integrals providing the inner product between field and spherical harmonics, see \cite[Sect. 2.1.2.]{bringout2016MPI}. 
In a second step, the expansion coefficients can then be incorporated into the MPI model that will be introduced in the next section. We believe that this calibration procedure is much less time consuming than the one used in an entirely measurement-based approach for MPI, in which a complete system matrix has to be measured \cite{Gruettner_etal2013}. Consequently, our approach can be considered as a hybrid reconstruction method in which parts of the model (in our case the generating magnetic fields) are measured.

An example taken from~\cite{bringout2016MPI} is shown in Fig.~\ref{fig:outerDriveConnectedSHSC}. While we have seen in Section \ref{sec:rotatingFFL} that an ideal $y$-drive field in the generation of a rotating FFL is represented by a single spherical harmonic, a corresponding realistic $y$-drive field contains a large number of higher degree spherical harmonics. These higher degree spherical harmonics have in return an impact on the low field volume of the generated magnetic field. While the ideal magnetic field generates a straight low field volume (see Fig.~\ref{subfig:LFV_Ideal}) a corresponding realistic low field volume has more the curved shape of a banana (Fig.~\ref{subfig:LFV_Real}).

\begin{figure}[h]
	\input{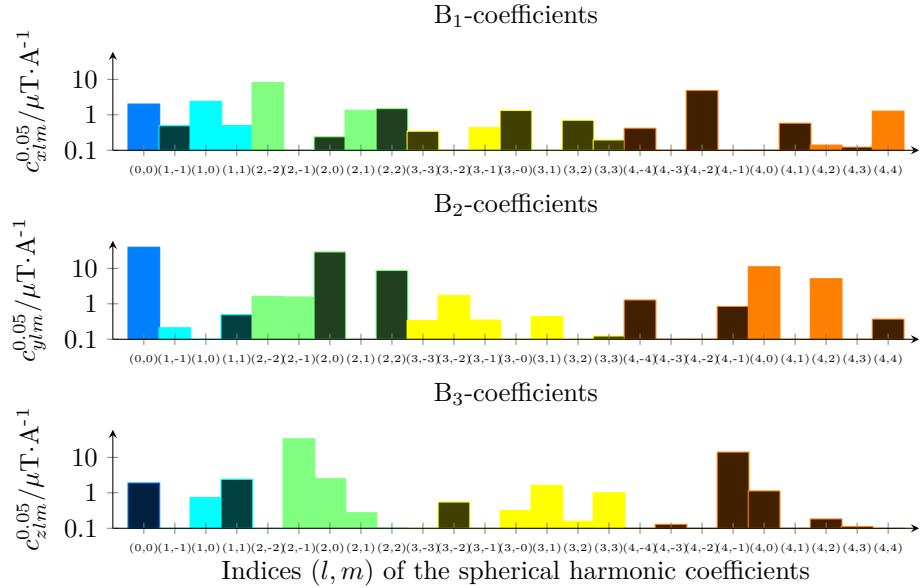}
	\caption[]{Spherical harmonic decomposition of an MPI $y$-drive coil obtained from the numerical model of a connected coil. The spherical harmonic coefficients $c_{l,m}^j$ are displayed up to degree $l=4$ and $-4 \leq m \leq 4$ in a sphere with radius $R=0.05$~m. Darker filling color indicates a negative coefficient.} 
	\label{fig:outerDriveConnectedSHSC}
\end{figure}

\section{A new 3D MPI model for realistic magnetic fields}
\label{sec:newMPImodel}

In this main section, we derive a new model for MPI that incorporates realistic field topologies and allows for a simple reconstruction at the same time. This model is based on the assumption that the applied magnetic field is parallel to its velocity field. This assumption is general enough to guarantee that realistic FFL-type field topologies are included. Further, we will show that the ideal 1D-FFP and FFL reconstruction formulas are special cases of this model.

\subsection{An imaging model for magnetic fields with parallel velocity field}

A time-dependent magnetic field $\vect{B}(\vect{r},t)$ is called \emph{parallel to its velocity field} if
\begin{equation} \label{eq:tangentialacceleration}
\vect{B}(\vect{r},t) \ \left\| \  \frac {\mathrm {d} }{\mathrm {d} t} \vect{B}(\vect{r},t) \right.
\end{equation}
is satisfied for all $\vect{r} \in \Omega$ and times $t$. In other words, at any position $\vect{r} \in \Omega$ and time $t$ the direction of
the velocity $\frac {\mathrm {d} }{\mathrm {d} t} \vect{B}(\vect{r},t)$ is pointing in the same or in the reversed direction of the magnetic field $\vect{B}(\vect{r},t)$. 
If $\vect{B}(\vect{r},t) \neq \vect{0}$ the parallelity \eqref{eq:tangentialacceleration} implies that the magnetic field $\vect{B}(\vect{r},t)$ does not change it's direction over time:
\[  \frac {\mathrm {d} }{\mathrm {d} t} \frac{\vect{B}(\vect{r},t)}{|\vect{B}(\vect{r},t)|}=  \frac{1}{|\vect{B}(\vect{r},t)|} \left( \frac {\mathrm {d} }{\mathrm {d} t} \vect{B}(\vect{r},t) - \frac{\vect{B}(\vect{r},t)}{|\vect{B}(\vect{r},t)|} \left\langle  \frac {\mathrm {d} }{\mathrm {d} t} \vect{B}(\vect{r},t), \frac{\vect{B}(\vect{r},t)}{|\vect{B}(\vect{r},t)|} \right\rangle  \right) = \vect{0}.
\]
In particular, in case of parallelity the acceleration of the magnetic field is also only performed tangentially in direction 
of the field $\vect{B}(\vect{r},t)$. Namely, we have
\[ \frac {\mathrm {d}^2 }{\mathrm {d} t^2} \vect{B}(\vect{r},t) = \frac{\vect{B}(\vect{r},t)}{|\vect{B}(\vect{r},t)|} \ \frac {\mathrm {d}^2 }{\mathrm {d} t^2} |\vect{B}(\vect{r},t)|. 
\]

\begin{theorem} \label{thm:1}
We assume that the function $\Mm: \Rr \to \Rr$ is odd and twice continuously differentiable and that the magnetic field $\vect{B}(\rrr,t)$ is differentiable and parallel to its velocity field. Then the time derivative of the magnetization $\Mmm(\vect{B}(\rrr,t))$ defined in
\eqref{eq:Magnetizationmodel1} can be simplified to
\begin{equation} \label{eq:Magnetizationmodel2}
\frac {\mathrm {d} }{\mathrm {d} t} \Mmm(\vect{B}(\rrr,t)) = \Mm'(|\vect{B}(\rrr,t)|) \, \frac {\mathrm {d} }{\mathrm {d} t} \vect{B}(\vect{r},t).
\end{equation}
In particular, the general MPI imaging model stated in \eqref{eq:faradays law} can be rewritten as
\begin{equation}
\label{eq:newmodel3D}
	u_{\nu}(t) = -\mu_0 \int_{\R^3} \Big\langle \vect{\rho}_{\nu}(\rrr), \frac {\mathrm {d} }{\mathrm {d} t} \vect{B}(\vect{r},t) \Big\rangle \;\Mm'(|\vect{B}(\rrr,t)|) \,  c(\rrr) \, \mathrm{d}\rrr.
\end{equation}
\end{theorem}

\begin{proof}
In the case that $\vect{B}(\rrr,t) \neq \vect{0}$, we use the chain rule to calculate the derivative of the vector field $\Mmm(\vect{B}(\rrr,t))$
given in \eqref{eq:Magnetizationmodel1}:
\begin{align*} \small
\frac {\mathrm {d} }{\mathrm {d} t} \Mmm(\vect{B}(\rrr,t))
=  \Mm'(|\vect{B}(\rrr,t)|) \left\langle  \frac {\mathrm {d} }{\mathrm {d} t} \vect{B}(\vect{r},t), \frac{\vect{B}(\vect{r},t)}{|\vect{B}(\vect{r},t)|} \right\rangle \frac{\vect{B}(\vect{r},t)}{|\vect{B}(\vect{r},t)|} + \Mm(|\vect{B}(\rrr,t)|) \frac {\mathrm {d} }{\mathrm {d} t} \frac{\vect{B}(\vect{r},t)}{|\vect{B}(\vect{r},t)|}.
\end{align*}
Since $\vect{B}(\rrr,t)$ is parallel to its velocity field $\frac {\mathrm {d} }{\mathrm {d} t} \vect{B}(\vect{r},t)$, we have 
$\frac {\mathrm {d} }{\mathrm {d} t} \frac{\vect{B}(\vect{r},t)}{|\vect{B}(\vect{r},t)|} = \vect{0}$ and $\left\langle  \frac {\mathrm {d} }{\mathrm {d} t} \vect{B}(\vect{r},t), \frac{\vect{B}(\vect{r},t)}{|\vect{B}(\vect{r},t)|} \right\rangle \frac{\vect{B}(\vect{r},t)}{|\vect{B}(\vect{r},t)|} = \frac {\mathrm {d} }{\mathrm {d} t} \vect{B}(\vect{r},t)$, and consequently
\begin{align*}
\frac {\mathrm {d} }{\mathrm {d} t} \Mmm(\vect{B}(\rrr,t))
=  \Mm'(|\vect{B}(\rrr,t)|) \, \frac {\mathrm {d} }{\mathrm {d} t} \vect{B}(\vect{r},t).
\end{align*}
In the case that $\vect{B}(\rrr,t) = \vect{0}$, we take a closer look to the univariate function $\Mm(x)/x$. Since $\Mm$ is odd and twice differentiable, we have
\[\frac{\Mm(x)}{x} \underset{x \to 0}{\longrightarrow} \Mm'(0) \quad \text{and} \quad \frac {\mathrm {d} }{\mathrm {d} x} \frac{\Mm(x)}{x} = 
\frac{\Mm'(x) x - \Mm(x)}{x^2} \underset{x \to 0}{\longrightarrow} \Mm''(0) = 0,
\]
and, thus, that the derivative of the function $\Mm(x)/x$ vanishes at $x = 0$. Therefore, we get also in the case $\vect{B}(\rrr,t) = \vect{0}$:
\begin{align*}
\frac {\mathrm {d} }{\mathrm {d} t} \Mmm(\vect{B}(\rrr,t))
=  \frac {\mathrm {d} }{\mathrm {d} t} \left( \frac{\Mm(|\vect{B}(\rrr,t)|)}{|\vect{B}(\rrr,t)|} \vect{B}(\rrr,t) \right) = \Mm'(0) \frac {\mathrm {d} }{\mathrm {d} t} \vect{B}(\rrr,t).
\end{align*}

\end{proof}

\subsection{The 1D-MPI imaging equation for a FFP along a line segment}
We show that the original 1D-MPI reconstruction formula for a 1D-FFP moving on an interval or a line segment in $\Rr^2$ can be deduced from Theorem \ref{thm:1}. In Subsection \ref{sec:FFP1D}, we already showed that  
the magnetic field $
\vect{B}(\rrr,t)
 = g ( -x, -y, 2z ) +  (d_x, d_y, d_z) \sin 2 \pi f_{\mathrm{d}} t
$
leads to the FFP $\rrr_{\mathrm{FFP}}(t) = \vect{v} \sin (2 \pi f_{\mathrm{d}} t)$ oscillating on the line segment $\mathbb{L}_{\vect{v}} = \{ \rrr = s \vect{v} \ | \ s \in [-1,1]\}$ in
direction $\vect{v} = (\frac{d_x}{g}, \frac{d_y}{g}, - \frac{d_z}{2g})$. For all points $\rrr$ in this line segment 
$\mathbb{L}_{\vect{v}}$, the magnetic field $\vect{B}(\rrr,t)$ is parallel to its velocity field $\frac {\mathrm {d} }{\mathrm {d} t} \vect{B}(\vect{r},t)$. Therefore, if the support $\supp c = \Omega \subset \Rr^3$ of the particle concentration is located in a close volume around the line segment $\mathbb{L}_{\vect{v}}$ we can use the simplified imaging 
equation \eqref{eq:newmodel3D} as a model for the MPI signal generation process. Using the gradient matrix $G = \diag(-g,-g,2g)$, we can write equation \eqref{eq:newmodel3D} as
\begin{equation*}
	u_{\nu}(t) = - 2 \pi \mu_0   f_{\mathrm{d}} \cos (2 \pi f_{\mathrm{d}} t) \int_{\Omega} \big\langle \vect{\rho}_{\nu}(\rrr), G \vect{v} \big\rangle \;\Mm'(|G(\rrr - \rrr_{\mathrm{FFP}}(t))|) \,  c(\rrr) \, \mathrm{d}\rrr.
\end{equation*}
Assuming that the coil sensitivity is constant $\vect{\rho}_{\nu}(\rrr) = \vect{\rho}_{\nu}$, this equation simplifies to the 1D-MPI model
\begin{equation}
\label{eq:newmodel3Db}
	u_{\nu}(t) = - 2 \pi \mu_0   f_{\mathrm{d}} \cos (2 \pi f_{\mathrm{d}} t) \big\langle \vect{\rho}_{\nu}, G \vect{v} \big\rangle 
	\Big( c \ast \Mm'(|G \cdot|) \Big) (\rrr_{\mathrm{FFP}}(t)),
\end{equation}
where $c \ast \Mm'(|G \cdot|)$ denotes the convolution of the particle concentration $c$ with the function $\Mm'(|G \cdot|)$. If the Langevin model of magnetization is used, we can further express $\Mm'$ in terms of the derivative \eqref{eq:Langevinmodelderivative} of the Langevin function. 

\begin{remark}
Although the convolution defined in \eqref{eq:newmodel3Db} is defined in terms of a three dimensional integral, the model in \eqref{eq:newmodel3Db} is conceptually a one-dimensional model for the reconstruction of the 
particle concentration $c$ along the line segment $\mathbb{L}_{\vect{v}}$. In a more idealized setting, we can also restrict the particle concentration $c$ to the line $\{\rrr = s \vect{v} \ | \ s \in \Rr\}$ and formulate the model \eqref{eq:newmodel3Db} in terms of a one dimensional convolution along this line. This 1D model was first formulated in \cite{Rahemeretal2009}. In \cite{Erbetal2018}, a profound mathematical analysis of the corresponding imaging operator was conducted. 
\end{remark}

\subsection{The 2D-MPI imaging equation for an ideal FFL}
Also the reconstruction formula for 2D-MPI imaging with an ideal FFL is a special case of Theorem \ref{thm:1}.
In Section \ref{sec:nonrotatingFFL}, the magnetic field 
\begin{align*}
\vect{B}(\rrr,t) =  2g \begin{pmatrix} -\big( \langle \vect{e}_{\alpha}, \rrr \rangle - \frac{d}{2g} \sin (2\pi f_{\mathrm{d}} t) \big) \sin \frac{\alpha}{2}  \\ \phantom{-}\big( \langle \vect{e}_{\alpha}, \rrr \rangle - \frac{d}{2g} \sin (2\pi f_{\mathrm{d}} t) \big) \cos \frac{\alpha}{2}  \\  z \end{pmatrix} \label{eq:FFL-norotation}
\end{align*}
was used to generate the non-rotating
$
\mathrm{FFL}(t) = \{ \rrr \ | \ \langle \vect{e}_{\alpha}, \rrr \rangle = \frac{d}{2g} \sin (2\pi f_{\mathrm{d}} t ) \}$.
Here, $\vect{e}_{\alpha} = (\sin \left( \frac{\alpha}{2} \right), - \cos \left( \frac{\alpha}{2} \right),0)$ denotes the normal vector of the FFL in the $xy$-plane.
The particular definition of $\vect{B}(\rrr,t)$ implies that for all points $\rrr$ in the $xy$-plane the magnetic field $\vect{B}(\rrr,t)$ is parallel to its velocity field $\frac {\mathrm {d} }{\mathrm {d} t} \vect{B}(\vect{r},t)$. Thus, if we assume that 
  the support $\supp c = \Omega \subset \Rr^3$ is a compact 2D region in the $xy$ plane, we can use the simplified imaging equation \eqref{eq:newmodel3D} for a model-based reconstruction. Inserting the magnetic field of the non-rotating FFL in the model equation \eqref{eq:newmodel3D} and using a simplified 2D integral over the domain $\Omega$ in the $xy$-plane, we get
  the formula
\begin{align*}
	u_{\nu}(t) &= - 2 \pi d \mu_0   f_{\mathrm{d}} \cos (2 \pi f_{\mathrm{d}} t) \! \int_{\Omega} \! \big\langle \vect{\rho}_{\nu}(\rrr), \vect{e}_{\alpha} \big\rangle \;\Mm'\left(2g \left| \langle \vect{e}_{\alpha}, \rrr \rangle - \textstyle \frac{d}{2g} \sin (2\pi f_{\mathrm{d}} t)\right| \right)  c(\rrr) \, \mathrm{d}\rrr.
\end{align*}
The vector $\vect{e}_{\alpha}^{\perp} = (\cos \frac{\alpha}{2}, \sin \frac{\alpha}{2},0)$ is perpendicular to $\vect{e}_{\alpha}$ in 
the $xy$ plane. With the basis vectors $\vect{e}_{\alpha}$ and $\vect{e}_{\alpha}^{\perp}$, we can write every point $\rrr$ in the $xy$-plane as $\rrr = s \vect{e}_{\alpha} + w \vect{e}_{\alpha}^{\perp}$. Using Fubini's theorem, we rewrite the bivariate integral above 
as the iterated integral
\begin{align*}
	u_{\nu}(t) &= - 2 \pi d \mu_0   f_{\mathrm{d}} \cos (2 \pi f_{\mathrm{d}} t) \!\! \int_{\Omega} \!\! \big\langle \vect{\rho}_{\nu}(s,w), \vect{e}_{\alpha} \big\rangle \, \Mm' \!\! \left(2g \left| s - \textstyle \frac{d}{2g} \sin (2\pi f_{\mathrm{d}} t)\right| \right)\! c(s,w) \d s \d w.
\end{align*}
Assuming that the coil sensitivity is constant $\vect{\rho}_{\nu}(\rrr) = \vect{\rho}_{\nu}$, this equation simplifies to 
\begin{align*}
	u_{\nu}(t) &= - 2 \pi d \mu_0   f_{\mathrm{d}} \cos (2 \pi f_{\mathrm{d}} t) \big\langle \vect{\rho}_{\nu}, \vect{e}_{\alpha} \big\rangle\int_{\Rr} \Mm'\left(2g \left| s - \textstyle \frac{d}{2g} \sin (2\pi f_{\mathrm{d}} t)\right| \right) \!\int_{\Rr} \! c(s,w) \d w \d s \\ &=
	- 2 \pi d \mu_0   f_{\mathrm{d}} \cos (2 \pi f_{\mathrm{d}} t) \big\langle \vect{\rho}_{\nu}, \vect{e}_{\alpha} \big\rangle\int_{\Rr} \Mm'\left(2g \left| s - \textstyle \frac{d}{2g} \sin (2\pi f_{\mathrm{d}} t)\right| \right) \mathcal{R}c (\vect{e}_{\alpha},s) \d s ,
\end{align*}
where $\mathcal{R}c (\vect{e}_{\alpha},s)$ denotes the Radon transform
of $c$ for the line given by $\langle \rrr, \vect{e}_{\alpha} \rangle = s$. For an ideal non-rotating \ac{FFL} we therefore get the imaging model
\begin{equation}
\label{eq:modelFFL}
	u_{\nu}(t) =
	- 2 \pi d \mu_0   f_{\mathrm{d}} \cos (2 \pi f_{\mathrm{d}} t) \big\langle \vect{\rho}_{\nu}, \vect{e}_{\alpha} \big\rangle  \Big(\Mm'( | 2g \cdot |) \ast \mathcal{R}c (\vect{e}_{\alpha},\cdot)\Big)(\textstyle \frac{d}{2g} \sin \big(2\pi f_{\mathrm{d}} t)\big),
\end{equation}
where $\ast$ denotes the standard one-dimensional convolution between
the kernel function $\Mm'(|2g s|)$ and the Radon transform $\mathcal{R}c (\vect{e}_{\alpha},s)$. 

\begin{remark} \label{rem:FBPreco}
Formula \eqref{eq:modelFFL} provides a direct way to reconstruct the 
particle concentration $c$ from the voltage signal $u_{\nu}$ \cite{Bente2014a,bringout2016MPI,Erbe2014,Knopp_etal2011hm}. Dividing the 
voltage signal $u_{\nu}$ by the velocity and sensitivity factor $- 2 \pi d \mu_0   f_{\mathrm{d}} \cos (2 \pi f_{\mathrm{d}} t) \big\langle \vect{\rho}_{\nu}, \vect{e}_{\alpha} \big\rangle$ and regridding the so obtained time signal onto the interval $[-\frac{d}{2g},\frac{d}{2g}]$, we get an expression for 
$(\Mm'(|2g \cdot|) \ast \mathcal{R}c (\vect{e}_{\alpha},\cdot))(s)$. Deconvolution then yields the Radon transform $\mathcal{R}c (\vect{e}_{\alpha},s)$ of the concentration $c$. By applying the filtered back projection (FBP) to the Radon data $\mathcal{R}c (\vect{e}_{\alpha},s)$ we finally can reconstruct $c$. Note that in some works, the deconvolution step is omitted in the reconstruction, see \cite{konkle2013twenty}.
\end{remark}

\begin{remark}For a rotating \ac{FFL} as given in Section \ref{sec:rotatingFFL}, the imaging equation \eqref{eq:newmodel3D} leads to the same formula
\eqref{eq:modelFFL}, with the only difference that the fixed angle $\alpha$ is replaced with the rotating angle $\alpha(t) = 2 \pi f_{\mathrm{rot}} t$. Note that in this case the parallelity assumption of Theorem \ref{thm:1} is not satisfied. However, if $f_{\mathrm{d}} >> f_{\mathrm{rot}}$, parallelity of the magnetic field $\vect{B}(\rrr,t)$ to its velocity field is almost given and the simplified imaging equation \eqref{eq:newmodel3D} provides a good approximation to the general imaging equation \eqref{eq:faradays law}. \end{remark}

\subsection{An approximative model for realistic magnetic fields}

If the magnetic field $\vect{B}(\rrr,t)$ does not provide an ideal FFP or FFL, we have in general not an analytic inversion formula for the reconstruction of the particle distribution $c$. Nevertheless, we can use the imaging equation \eqref{eq:newmodel3D} to derive a discrete model-based MPI equation in the case that the magnetic fields are parallel to its velocity field. This allows us also for more complex magnetic fields to reconstruct the particle density $c$ algebraically from a modelled system matrix.

As a first step towards a discretization of the integral in \eqref{eq:newmodel3D}, we approximate 
the derivative $\Mm'(x)$ for $|x| < b $ using a piecewise constant function. The function $\Mm'$ is in general localized in a small region around the origin. In this way, $\Mm'(|\vect{B}(\rrr,t)|)$ gets essentially large only in the low-field volume (LFV) of the magnetic field $\vect{B}(\rrr,t)$, i.e., in those regions in which the modulus $|\vect{B}(\rrr,t)|$ is small. The chosen threshold $b>0$ therefore gives a bound for the \ac{LFV} of $\vect{B}(\rrr,t)$ in which $\Mm'$ is large enough to give an impact for the integral equation \eqref{eq:newmodel3D}. 

In the following, we restrict our attention to the approximation of the function $\Mm'$ in the positive interval $[0,b)$. We consider $N$ nodes $0 < x_1 < \cdots < x_N < b$ and set $x_{N+1} = b$, $x_0 = 0$. The approximation of $\Mm'$ with piecewise constant functions is defined on the intervals $I_n = [x_{n},x_{n+1})$, $n = 0, \ldots, N$. We construct
\begin{equation} \label{eq:approxmagnetization}
\Mm_{N}'(x) = \sum_{n=0}^N s_n \chi_{I_n}(|x|) 
\end{equation}
in such a way that $\lim_{N\to \infty} \Mm_{N}'(x) = \Mm'(x)$ for all $x \in [0,b)$. Two schemes to obtain the nodes $x_n$ and the values $s_n$, $n = 0, \ldots, N$, are given in Section \ref{sec:numerics} (we will use the Langevin function \eqref{eq:Langevinmodel2} to model $\Mm$). This construction allows us to approximate the time-derivative of the magnetization derived in \eqref{eq:Magnetizationmodel2} with a piecewise constant function:
\begin{align*}
\Mm_N'(|\vect{B}(\rrr,t)|) \, \frac {\mathrm {d} }{\mathrm {d} t} \vect{B}(\vect{r},t) &=  \left\{ 
\begin{array}{ll}  \displaystyle s_{n} \, \frac {\mathrm {d} }{\mathrm {d} t}\vect{B}(\rrr,t) & \text{if}\; |\vect{B}(\rrr,t)| \in I_n, \\[2ex] 
\vect{0} & \text{if}\;|\vect{B}(\rrr,t)| \geq b
\end{array} \right. \\[1ex]
&= \sum_{n=0}^{N} s_{n}\, \chi_{I_n}(|\vect{B}(\rrr,t)|) \, \frac {\mathrm {d} }{\mathrm {d} t} \vect{B}(\rrr,t).
\end{align*}

\noindent Using the approximate derivative $\Mm_N'$ instead of $\Mm'$, the imaging equation given in equation \eqref{eq:newmodel3D} can be written as
\begin{equation}
\label{eq:faradayslawapprox}
	u_{\nu}(t) = -\mu_{0} \sum_{i=0}^{N} s_{n} \int_{F_{n}(t)} \left\langle \vect{\rho}_{\nu}(\rrr), \frac {\mathrm {d} }{\mathrm {d} t} \vect{B}(\rrr,t) \right\rangle\; c(\rrr) d\rrr,
\end{equation}
where
\begin{equation}
	F_{n}(t) = \{\rrr\in \Omega \ | \ |\vect{B}(\rrr,t)| \in I_n\}.
\end{equation}

\noindent Introducing the kernel functions
\[K_{\nu}(\rrr,t) =  -\mu_{0}  \left\langle\vect{\rho}_{\nu}(\rrr), \frac {\mathrm {d} }{\mathrm {d} t} \vect{B}(\rrr,t) \right\rangle,\]
we finally obtain the  integral equation
 \begin{equation} \label{eq:approximateMPImodel}
u_{\nu}(t) =  \sum_{n=0}^{N} s_{n} \int_{F_{n}(t)} K_{\nu}(\rrr,t) c(\rrr) d\rrr.
 \end{equation}
 
\noindent We have the following limiting relation between the approximate model \eqref{eq:approximateMPImodel} and the orignal equation \eqref{eq:newmodel3D}.

\begin{theorem}
Let $\supp c = \Omega \subset \Rr^3$ be a compact set, the function $K_\nu(\rrr,t) c (\rrr)$ be integrable over $\Omega$ and $\Mm$ be twice continuously differentiable. Set $b > 0$ such that 
\[ \max_{t \in [0,T_0]} \max_{\vect{r} \in \Omega} |B(\vect{r},t)| < b.\]
Let $\Mm_N'$ be an approximation of $\Mm'$ given in \eqref{eq:approxmagnetization} such that $\lim_{N\to \infty} \Mm_{N}'(x) = \Mm'(x)$ uniformly on $[0,b)$. Then, for $t \in [0,T_0]$, we have
\[
\lim_{N \to \infty} \sum_{n=0}^{N} s_{n} \int_{F_{n}(t)} K_{\nu}(\rrr,t) c(\rrr) d\rrr = \int_{\Omega} \Mm'(|B(\vect{r},t)|) K_{\nu}(\rrr,t) c(\rrr) d\rrr.\]
\end{theorem}

\begin{proof}
If $\Mm$ is twice continuously differentiable, Theorem \ref{thm:apprerror} below together with one of the node selection strategies given in Section \ref{sec:selectionofnodes} guarantees the existence of a piecewise uniform approximation $\Mm_N'$ of $\Mm'$ on the interval $[0,b)$. In fact, Theorem \ref{thm:apprerror} shows that both constructions given in Section \ref{sec:piecewiselinear} are adequate.
For $t \in [0,T_0]$, the domain $\Omega$ corresponds to the disjoint union $\bigcup_{n=0}^N F_n(t)$ and the uniform convergence of $\Mm_N'$ yields 
\begin{align*}
\lim_{N \to \infty} \sum_{n=0}^{N} s_{n} \int_{F_{n}(t)} K_{\nu}(\rrr,t) c(\rrr) d\rrr &= \lim_{N \to \infty} \int_{\Omega} 
\sum_{n=0}^{N} s_{n} \chi_{I_n}(|\vect{B}(\vect{\rrr},t)|) K_{\nu}(\rrr,t) c(\rrr) d\rrr
\\ &= \int_{\Omega} \Mm'(|B(\vect{r},t)|) K_{\nu}(\rrr,t) c(\rrr) d\rrr.
\end{align*}
\end{proof}

We will use the approximate imaging equation given in \eqref{eq:approximateMPImodel} to obtain a model-based algebraic reconstruction formula. We can regard \eqref{eq:approximateMPImodel} however also as an approximative MPI imaging model in case that the magnetic field $\vect{B}(\rrr,t)$ is, as in Theorem \ref{thm:1}, parallel to its velocity field.  
\section{Numerical implementation of the new model} \label{sec:numerics}

\subsection{Piecewise linear approximation of the magnetization function $\Mm$} \label{sec:piecewiselinear}

In this section, we shortly provide two ways on how to approximate the magnetization function $\Mm$ with a piecewise linear function $\Mm_N$, and, at the same time, on how to approximate the derivative $\Mm'$ with a piecewise constant function $\Mm_N'$. The approximation $\Mm_N$ on the positive half-axis consists of a polygon with $N+2$ linear polynomials. For this, we consider $N$ nodes $0 < x_1 < \cdots < x_N < b$ and set $x_{N+1} = b$, $x_0 = 0$. The linear polynomials are defined on the intervals $I_n = [x_{n},x_{n+1})$, $n = 0, \ldots, N$, and on 
$I_{N+1} = [x_{N+1}, \infty)$.

\begin{itemize}
 \item[1)] (Secant approximation scheme) For $n \in \{0, \ldots, N\}$, set
\[ s_n = \frac{\Mm(x_{n+1})-\Mm(x_{n})}{x_{n+1}-x_{n}}.\]
 \item[2)] (Tangent approximation scheme) For $n \in \{0, \ldots, N\}$, set
\[ s_n = \left\{ \begin{array}{ll} \Mm'(0) & \text{if $n = 0$}, \\ 
\Mm'(\frac{x_{n+1} + x_n}{2}) & \text{if $n > 0$}.\end{array} \right.\]
\end{itemize}

\noindent For both choices, we get for $x \geq 0$ the following approximants for the magnetization $\Mm$ and its derivative $\Mm'$:
\begin{align*}
\Mm_{N}'(x) = \sum_{n=0}^N s_n \chi_{I_n}(x), \quad \Mm_{N}(x) = \int_{0}^{x} \sum_{n=0}^N s_n \chi_{I_n}(y) dy.
\end{align*}
An illustration of the approximation $\Mm_N$ for the secant scheme is shown in Fig. \ref{fig:langevinapproximation2}. For negative $x$, we expand $\Mm_N$ and $\Mm_N'$ such that $\Mm_N$ is odd and $\Mm'$ is even on $\Rr$, i.e., 
\[ \Mm_{N}'(x) = \Mm_{N}'(-x), \qquad \Mm_{N}(x) = - \Mm_{N}(-x) \qquad \text{for} \; x \in \Rr.\] 
 
\begin{example}
For the tangent scheme, we obtain a simple approximation of the derivative $\Mm'$ already for $N=0$. In this case, using the Langevin model \eqref{eq:Langevinmodel2} to describe $\Mm$ and formula \eqref{eq:Langevinmodelderivative} for the derivative, we get the approximation
\begin{equation}
	\Mm_N'(x) = \begin{cases}
	\displaystyle
		\frac{m_{0}\lambda}{3} , &0 \leq x< b,\\[1ex]
		0, & x \geq b.
	\end{cases}
\end{equation}
Then, if \[ F_0(t) = \{ \rrr \in \R^3: \; |\vect{B}(\rrr,t)| < b\} \]
denotes the low field volume (LFV) at time $t$ in which the modulus of the magnetic field is smaller than $b$, we obtain in \eqref{eq:approximateMPImodel}
the simple integral equation
\begin{equation*}
	u_{\nu}(t) = -\frac{\mu_{0}m_{0}\lambda}{3} \int_{F_0(t)} \left\langle \vect{\rho}_{\nu}(\rrr),  \frac{\mathrm d}{\mathrm{d} t} \vect{B}(\rrr,t) \right\rangle \; c(\rrr)\, d\rrr, 
\end{equation*}
i.e., the voltage signal $u_{\nu}$ is approximately generated by integrating the particle density together with the velocity term $\left\langle \vect{\rho}_{\nu}(\rrr),  \frac{\mathrm d}{\mathrm{d} t} \vect{B}(\rrr,t) \right\rangle$ over the 
LFV $F_0(t)$.
\end{example}

\begin{figure}
\centering
\includegraphics[width=9cm]{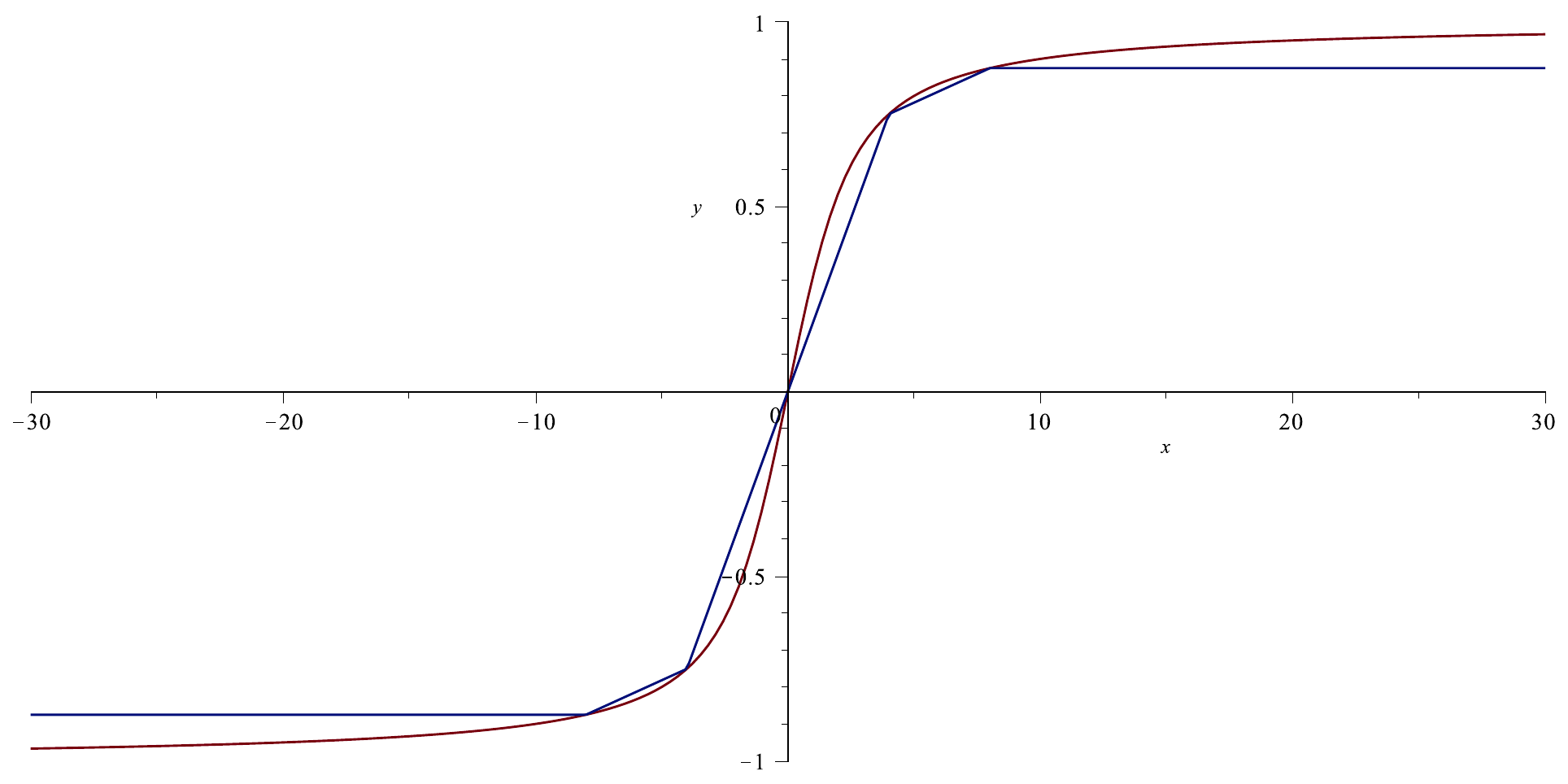}
\caption{Piecewise approximation of the Langevin function $L$ by the secant scheme.}
\label{fig:langevinapproximation2}
\end{figure}

\begin{theorem} \label{thm:apprerror} If $\Mm$ is twice continuously differentiable then, for both, the tangent and the secant approximation scheme, we have the properties:
\begin{align} \sup_{x \in [0,b)} |\Mm'(x) - \Mm_N'(x)| &\leq \sup_{x \in [0,b)}|\Mm''(x)| \max_{k \in \{0, \ldots, N\}} ( x_{k+1} - x_{k}), \label{eq:error1}\\
\sup_{x \in [0,b)} |\Mm(x) - \Mm_N(x)| &\leq \sup_{x \in [0,b)}|\Mm''(x)| \ \sum_{k=0}^N ( x_{k+1} - x_{k})^2. \label{eq:error2}
\end{align}
\end{theorem}

\begin{proof}
For both, the tangent and the secant scheme, the mean value theorem on the interval $I_k$ provides the estimates
\[|\Mm'(x) - \Mm_N'(x)| \leq \sup_{x \in I_k }|\Mm''(x)| (x_{k+1} - x_k). \]
This immediately implies \eqref{eq:error1}. We can use this estimate also to obtain \eqref{eq:error2}:
\begin{align}  |\Mm(x) - \Mm_N(x)| &= \left|\int_0^x (\Mm'(x) - \Mm_N'(x)) \d x \right| \leq \sum_{k=0}^N \int_{I_k} |\Mm'(x) - \Mm_N'(x)| \d x \notag \\ &\leq \sup_{x \in [0,b)}|\Mm''(x)| \ \sum_{k=0}^N ( x_{k+1} - x_{k})^2. \label{eq:error3}
\end{align}
\end{proof}

\subsection{Selection of the nodes} \label{sec:selectionofnodes}

The nodes $x_1, \ldots, x_N$ in the interval $(0,b)$ can be chosen in several different ways. Based on our result in Theorem \ref{thm:apprerror}, we provide two simple options:

\begin{itemize}
 \item[1)] Equidistant points (the simplest choice): 
\[ x_k = \frac{b \, k}{N+1}, \qquad k \in \{0, \ldots, N+1\}.\]
 \item[2)] $L_1$-optimal points: choose $x_1, \ldots, x_N$ in $(0,b)$ such that the $L^1$-norm
 \[ \int_{0}^b |\Mm'(x) - \Mm_N'(x)| \d x\]
 is minimized. 
\end{itemize}

\noindent Theorem \ref{thm:apprerror} ensures that both choices lead to a uniform convergence of $\Mm_N'$ towards $\Mm'$ on the interval $[0,b)$. The $L_1$-optimized variant yields a better approximation quality for the derivative $\Mm'$ as well as for the magnetization function $\Mm$. This can  be seen particularly in the estimate \eqref{eq:error3}, in which an $L_1$- optimal ensemble makes the second inequality redundant. 

In case of the tangent scheme, we give an explicit formula for the $L_1$-norm. We assume that the derivative $\Mm'$ is positive and strictly monotonically decreasing when $x \geq 0$. This is indeed the case if the magnetization $\Mm$ is given as in \eqref{eq:Langevinmodel2} by the Langevin model. Then, we get the explicit formula\begin{align*}
\int_{0}^b |\Mm'(x) - \Mm_N'(x)| \d x &= \sum_{k = 0}^N \int_{I_k} |\Mm'(x) - s_k| \d x \\
&= \sum_{k = 0}^N \left( \int_{x_k}^{\frac{x_k + x_{k+1}}{2}} \Mm'(x) \d x - \int_{\frac{x_k + x_{k+1}}{2}}^{x_{k+1}} \Mm'(x) \d x \right) \\
&= \sum_{k = 0}^N \left( 2 \, \Mm( \textstyle \frac{x_k + x_{k+1}}{2}) - \Mm(x_{k+1}) - \Mm(x_k) \right) = F(x_1, \ldots, x_N).
\end{align*}
Thus, in order to find the ensamble in which the $L^1$-norm gets minimal, we only have to minimize the functional $F$. 

\subsection{Full discretization and model-based imaging matrix}
We assume to have $T$ discrete time measurements $\vect{u}_{\nu}=(u_{\nu}(t_{1}),\dots,u_{\nu}(t_{T}))$ of the voltage signal $u_{\nu}$. To discretize the particle density $c(\rrr)$, we use the representation
\[c(\rrr)=\sum_{k=1}^{K}c_{k}\delta_{k}(\rrr)\]
of $c$ in a given set of basis functions $\delta_{k}$, $k = 1, \ldots, K$.
In our implementation, we use a pixel basis, i.e., $\delta_{k}=\chi_{Q_{k}}$ for rectangular pixel regions $Q_{k}$ centered at the pixel locations $\rrr_{k}$. Therefore, we have $c_{k}=c(\rrr_{k})$.

\noindent The approximate model equation stated in \eqref{eq:faradayslawapprox} can now be discretized as
\[u_{\nu}(t_{j})=\sum_{n=0}^{N} s_{n} \sum_{k=1}^{K} c_{k} \int_{F_{n}(t_{j})} K_{\nu}(\rrr,t_{j}) \delta_{k}(\rrr) d\rrr.\]

\noindent We denote by $S_{\nu,n} \in \Rr^{T \times K}$ the rectangular matrix with the entries 
\[(S_{\nu,n})_{j,k}= \int_{F_{n}(t_{j})} K_{\nu}(\rrr,t_{j}) \delta_{k}(\rrr) d\rrr = \int_{F_{n}(t_{j})\cap Q_{k}} K_{\nu}(\rrr,t_{j}) d\rrr.\]
Then, a model-based discrete imaging equation to recover the particle concentration can be written as
\begin{equation} \label{eq:discretizedmodelequation}
\vect{u}_{\nu} = \left(\sum_{n=0}^{N} s_{n} S_{\nu,n}\right) \vect{c} = S_{\nu} \vect{c},
\end{equation}
where $\vect{c}=(c_{1},\dots,c_{K})$ are the sought concentration values and $S_{\nu} \in \Rr^{T \times K}$ is the modeled system matrix for the receive coil $\nu$.

\subsection{Algebraic reconstruction of the particle concentration}
In order to solve the system of equations \eqref{eq:discretizedmodelequation}
for the particle concentration $\vect{c}$, we combine the information of all $V$ receive coils. In this way, we get the system
\[ \underbrace{\begin{pmatrix} \vect{u}_1 \\ \vdots \\ \vect{u}_V \end{pmatrix}}_{\vect{u}} = \underbrace{\begin{pmatrix} S_1 \\ \vdots \\ S_V \end{pmatrix}}_{S}  \vect{c}.\]
From this system we extract the particle concentration $\vect{c}$ by calculating the solution of the normal equation $S^T S \vect{c} = S^T \vect{u}$ iteratively using the LSQR algorithm together with an early stopping rule. 

\begin{remark}
Although the particle reconstruction in MPI is in general an ill-posed inverse problem \cite{Erbetal2018,KluthJinLi2018,marz2016model}, at that stage, we did not incorporate additional regularization for the solution of the linear system $\vect{u} = S \vect{c}$. The various discretization steps applied in our model and the early stopping of the LSQR procedure already provide a certain regularization of the problem. In combination with our model, one could of course apply also more advanced regularization schemes as for instance described in \cite{storath2017edge}.
\end{remark}
\section{Experiments} \label{sec:experiments}

Based on the phantom presented in Fig.~\ref{subfig:phantom}, two types of magnetic field topologies are used to study the influence on the reconstruction: an ideal rotating FFL magnetic field, represented by a few low-degree spherical harmonics, as described in Section \ref{sec:rotatingFFL}; and a realistic one, obtained from either a realistic numerical model of the magnetic coil or from measurements. A concrete example of such a realistic field is given in Section \ref{sec:realistidfield}.

\begin{figure}[H]
	\subfloat[]{%
	 \includegraphics[width=0.32\textwidth]{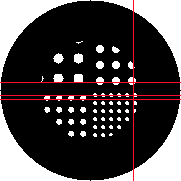}
	 \label{subfig:phantom}} 
	\subfloat[]{%
		\includegraphics[width=0.32\textwidth]{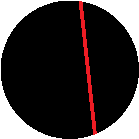}
		\label{subfig:LFV_Ideal}} 
	\subfloat[]{%
		\includegraphics[width=0.32\textwidth]{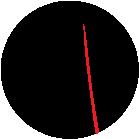}
		\label{subfig:LFV_Real}} 
	 
	\caption[]{
	Ideal and realistic rotating FFL's:
	\subref{subfig:phantom} A phantom composed of circles (in white) of 4, 6, 8 and 10~mm diameter inside a 100~mm diameter circle. The reconstructions along the illustrated horizontal and vertical lines are shown in Fig.~\ref{fig:cut_threshold}, \ref{fig:cut_steps} and \ref{fig:cut_discret}. \subref{subfig:LFV_Ideal} and \subref{subfig:LFV_Real} show the \ac{LFV} (in red) with a field amplitude smaller than 2~mT for an ideal and a realistic field topology, respectively ($f_{\mathrm{rot}}$=100~Hz).}
	\label{fig:FFL_Phantom_Banana_2}
\end{figure}

The influence of the higher degree spherical harmonics on the \ac{LFV} can be observed by comparing Fig.~\ref{subfig:LFV_Ideal} which illustrates the \ac{LFV} at 2~mT for ideal topologies at a given time $t=17.25~\mu$s in our sequence, with Fig.~\ref{subfig:LFV_Real} which shows the corresponding \ac{LFV} for the realistic field topology used afterwards in our simulations. The \ac{LFV} compared to an ideal \ac{FFL} is slightly bended and interrupted on the upper part. To highlight the effect of these differences on the reconstructed particle concentrations, we run now experiments comparing side by side the images obtained either by a \ac{FBP} for ideal and realistic fields or by our algebraic method.
\subsection{General experimental parameters}

For all the presented results, the drive field frequency is fixed at $f_{\mathrm{d}}$=25~kHz and the line rotation frequency $f_{\mathrm{rot}}$ is varied from 100~Hz to 1000~Hz. Due to a RAM limitation of 1~TB in our system, a rotation frequency of 10~Hz has not be conducted. Thus, 250 and 25 projections were used for the 100~Hz and 1000~Hz reconstruction, respectively. It is common to set the drive frequency between <1~kHz to 150~kHz and the rotation frequency from <1~Hz to 100~Hz. The sampling frequency is set to 8~Mhz (4~Mhz for FBP), thus obtaining 160 points per projection (80 for FBP) and emulating the common properties of the acquisition hardware used. Note that in order to do a frequency filtering on the measured signal, the whole rotation was always simulated to obtain a perfectly resolved spectra. Indeed, the hardware of an actual MPI scanner always filters out a frequency range around $f_{\mathrm{d}}$. To reproduce this effect, we always removed all the information below 1.4~$f_{\mathrm{d}}$ from the measured signal. Furthermore, a Gaussian noise was added to all simulations. For the signal simulation, a spatial discretization of $1\times1\times1$~mm$^3$ was used, whereas a discretization of $1.3\times1.3\times1$~mm$^3$ was used for the system matrix. The solution of the proposed discretized MPI model \eqref{eq:discretizedmodelequation} was solved using Matlab's LSQR implementation, which was always stopped after $20$ iterations. This has been optimized on a simulation using a rotation frequency $f_{\mathrm{rot}}$ of $100$~Hz, a threshold of $b = 10$~mT and kept constant for all further tests.

\subsection{Comparison with the filtered back projection}

\begin{figure}[H]
	\subfloat[]{%
		\includegraphics[width=0.48\textwidth,valign=c]{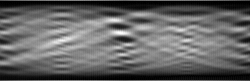}
		\label{subfig:sino_normal}} 
	\subfloat[]{%
		\includegraphics[width=0.48\textwidth,valign=c]{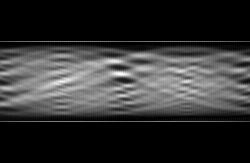}
		\label{subfig:sino_padded}}	 
	\caption[]{Two sinograms of the phantom from Fig.~\ref{subfig:phantom} used as input data for the algorithms. \subref{subfig:sino_normal} Sinogram of a phantom in which the maximal displacement of the Radon transform is adapted to the support of the phantom. It is used to generate the reconstruction shown in Table~\ref{tab:reco_normal_sino}. \subref{subfig:sino_padded} Same sinogram padded with zero, to increase the reconstruction area of the FBP.}
	\label{fig:FFL_FBP}
\end{figure}

We compare our method with the filtered back projection (FBP), which is commonly used for \ac{FFL} systems to perform the image reconstruction~\cite{Bente2014a,konkle2013twenty}. The \ac{FBP} implementation of Matlab (Version 7.11.0) was used to performed the first test. The Radon projections are obtained by the reconstruction steps described in Remark \ref{rem:FBPreco}, using half a period of a drive field sweep with frequency $f_{\mathrm{d}}$ for fixed approximate angles $\alpha \approx 2 \pi f_{\mathrm{rot}} t$ and for discrete $s =\frac{d}{2g} \sin \big(2\pi f_{\mathrm{d}} t)$ in a subinterval of $[-\frac{d}{2g},\frac{d}{2g}]$. The so obtained Radon data is assembled into a sinogram and then reconstructed using the \ac{FBP}.

The sinogram obtained using the ideal model of a rabbit sized FFL MPI scanner~\cite{bringout2016MPI} is presented in Fig. \ref{subfig:sino_normal}. To fully asses the differences between the reconstruction methods over the whole scanner opening, the sinogram is further padded with zeros, as shown in Fig.~\ref{subfig:sino_padded}.

The results of the \ac{FBP} are compared with the images obtained by our method. The system matrix $S_{\nu}$ for the algebraic reconstruction in \eqref{eq:discretizedmodelequation} is constructed using the threshold $b = 10~\mathrm{mT}$ and the secant approximation scheme with $N = 30$ equidistant nodes for the discretization of the Langevin function. The information about the magnetic field $\vect{B}(\rrr,t)$ and its time-derivative is obtained by measurements or simulations. 

Tables~\ref{tab:reco_normal_sino} and~\ref{tab:reco_0padded_sino} highlight the main advantages of our method. Indeed, the presented model-based reconstruction method compensates the main artefacts introduced by the idealized assumptions in the \ac{FBP} reconstruction. Looking at the Table~\ref{tab:reco_normal_sino}, the rotation artifacts which appeared in all images produced by the \ac{FBP} independently of the field complexity and only linked to the continuous rotation of the \ac{LFV} during an acquisition are compensated. The distortion artifacts visible on both Tables, introduced by the complex field topology are also corrected. This can be further observed by the corrected distances between two points of the phantom.

\def\arraystretch{1}
\setlength{\tabcolsep}{1pt}
\begin{table}[H]
	\begin{tabular}{cccc}
		 & Ideal FFL and FBP& Realistic FFL and FBP & Realistic FFL \& our method\\
		\rotatebox[origin=c]{90}{100 Hz} &
		\parbox[c]{0.32\textwidth}{\includegraphics[width=0.32\textwidth]{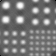}} &
		\parbox[c]{0.32\textwidth}{\includegraphics[width=0.32\textwidth]{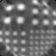}} &
		\vspace{1pt}\parbox[c]{0.32\textwidth}{\includegraphics[width=0.32\textwidth]{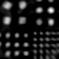}}\vspace{1pt} \\
		\rotatebox[origin=c]{90}{1000 Hz} &
		\parbox[c]{0.32\textwidth}{\includegraphics[width=0.32\textwidth]{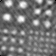}} &
		\parbox[c]{0.32\textwidth}{\includegraphics[width=0.32\textwidth]{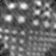}} &
		\vspace{1pt}\parbox[c]{0.32\textwidth}{\includegraphics[width=0.32\textwidth]{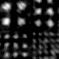}} \\
	\end{tabular}
	\caption{Comparison of reconstructions for different rotating FFL topologies using $f_{\mathrm{rot}} = 100$ Hz (1. row),  $f_{\mathrm{rot}} = 1000$ Hz (2. row), and identically chosen fields of view.}
	\label{tab:reco_normal_sino}
\end{table}

\vspace{-5mm}
\begin{table}[H]
	\begin{tabular}{cccc}
		 & Ideal FFL and FBP& Realistic FFL and FBP & Realistic FFL \& our method\\
		\rotatebox[origin=c]{90}{100 Hz} &
		\parbox[c]{0.32\textwidth}{\includegraphics[width=0.32\textwidth]{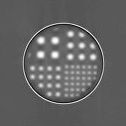}} &
		\parbox[c]{0.32\textwidth}{\includegraphics[width=0.32\textwidth]{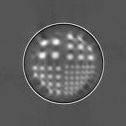}} &
		\vspace{1pt}\parbox[c]{0.32\textwidth}{\includegraphics[width=0.32\textwidth]{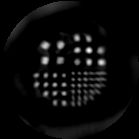}}\vspace{1pt} \\
	\end{tabular}
	\caption{Reconstructions of Table \ref{tab:reco_normal_sino} (first row) using an enlarged field of view. A sinogram padded with zero, as illustrated in Fig.  \ref{subfig:sino_padded}, is used to extend the reconstruction area of the FBP.}
	\label{tab:reco_0padded_sino}
\end{table}

\subsection{Influence of amplitude threshold}

We further study the influence of the threshold $b$ on the reconstruction quality. In Fig.~\ref{fig:SysMat_Thr} we can see how this threshold determines the volume of the \ac{LFV} used in the discrete imaging equation \eqref{eq:discretizedmodelequation} to model the generation of the MPI signal.

\vspace{-6mm}

\begin{figure}[H]
    \centering
	\subfloat[]{%
		\begin{tikzpicture}

\begin{axis}[%
width=0.40\textwidth,
unit vector ratio=1 1 1,
axis on top,
scale only axis,
xmin=0.5,
xmax=139.5,
xtick=\empty,
y dir=reverse,
ymin=0.5,
ymax=139.5,
ytick=\empty,
name=plot2,
colorbar,
colormap/blackwhite,
point meta min=-1,
point meta max=1
]
\addplot [forget plot] graphics [xmin=0.5,xmax=139.5,ymin=0.5,ymax=139.5] {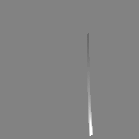};
\end{axis}

\end{tikzpicture}%
	 	\label{subfig:SM_2mT}} 
	\subfloat[]{%
		\begin{tikzpicture}

\begin{axis}[%
width=0.40\textwidth,
unit vector ratio=1 1 1,
axis on top,
scale only axis,
xmin=0.5,
xmax=139.5,
xtick=\empty,
y dir=reverse,
ymin=0.5,
ymax=139.5,
ytick=\empty,
name=plot2,
colorbar,
colormap/blackwhite,
point meta min=-1,
point meta max=1
]
\addplot [forget plot] graphics [xmin=0.5,xmax=139.5,ymin=0.5,ymax=139.5] {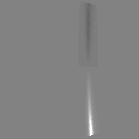};
\end{axis}

\end{tikzpicture}%
		\label{subfig:SM_10mT}} 
	\caption[]{Influence of the threshold $b$ on the information included in the system matrix $S_{\nu}$ for $f_{\mathrm{rot}}$=100~Hz at t=17.25~$\mu$s using the secant approximation scheme with $N = 30$ equidistant nodes. The normalized entries of the system function $S_{\nu}$ are displayed: \subref{subfig:SM_2mT} for a threshold of $b=2$~mT and \subref{subfig:SM_10mT} for a threshold of $b = 10$ mT.}
	\label{fig:SysMat_Thr}
\end{figure}

\vspace{-6mm}

\begin{figure}[H]
	\subfloat[]{%
		\includegraphics[width=0.32\textwidth]{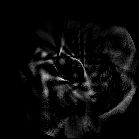}
		\label{subfig:Thresholds_1mT}} 
	\subfloat[]{%
		\includegraphics[width=0.32\textwidth]{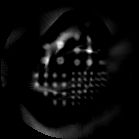}
		\label{subfig:Thresholds_2mT}} 
	\subfloat[]{%
	 	\includegraphics[width=0.32\textwidth]{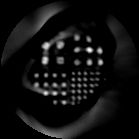}
	 	\label{subfig:Thresholds_3mT}} 
	 
	\subfloat[]{%
		\includegraphics[width=0.32\textwidth]{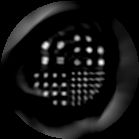}
		\label{subfig:Thresholds_4mT}} 
	\subfloat[]{%
		\includegraphics[width=0.32\textwidth]{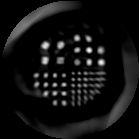}
		\label{subfig:Thresholds_5mT}} 
	\subfloat[]{%
	 	\includegraphics[width=0.32\textwidth]{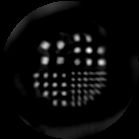}
	 	\label{subfig:Thresholds_10mT}} 
	\caption[]{Reconstruction with thresholds 1, 2, 3, 4, 5 and 10~mT shown in \subref{subfig:Thresholds_1mT} to \subref{subfig:Thresholds_10mT}.}
	\label{fig:FFL_thresholds_variation}
\end{figure}

\begin{figure}[H]
	\input{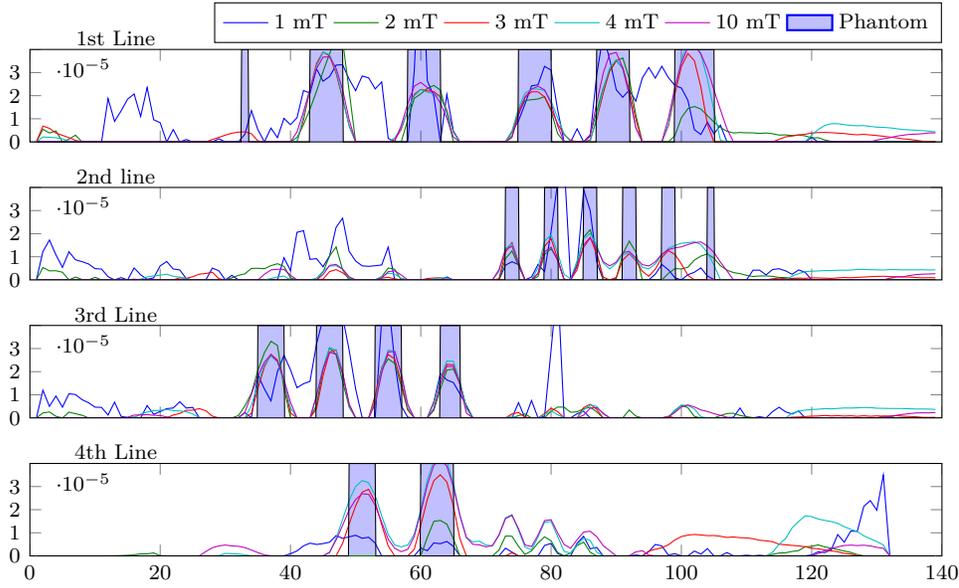}
	\caption[]{Comparison of the line profiles (as shown in Fig.~\ref{subfig:phantom}) for the reconstructions obtained with our method using different thresholds $b$ and the original phantom.}
	\label{fig:cut_threshold}
\end{figure}

\noindent The reconstructions in Fig. \ref{fig:FFL_thresholds_variation} and the line profiles in Fig.~\ref{fig:cut_threshold} show that the choice of the threshold $b$ has a strong impact on the reconstruction quality in the entire \ac{FOV} when $b$ is in the range 1 to 4~mT. On the other hand, if $b$ is between 4 and 10~mT only small differences are visible on the periphery of the \ac{FOV}. Thus, in this example a threshold $b$ of 10~mT is sufficient, and a threshold of 4~mT yields already very good results for the central part of the \ac{FOV}. Note that a smaller threshold is desirable from a computational point of view in order to benefit from a sparser representation of the system matrix $S_{\nu}$.

\subsection{Discretization effects}

In two additional tests, we study discretization effects on the reconstruction. In the first test, we search for the optimal number of equidistant nodes $N$ for the piecewise approximation of the Langevin function in the tangential approximation scheme on an example with $f_{\mathrm{rot}}$=100~Hz and a threshold $b = 10$~mT. The corresponding results are illustrated in Fig. \ref{fig:FFL_steps_variation} and Fig. \ref{fig:cut_steps}. It is visible that already for $N = 8$ the piecewise approximation of the Langevin function provides acceptable reconstructions.  

In a second experiment, we test three different discretization techniques for the Langevin function on an example with $f_{\mathrm{rot}}$=100~Hz, threshold $b = 10$~mT and $N = 30$ nodes. As shown in Figure \ref{fig:FFL_discre_variation} and Fig. \ref{fig:cut_discret}, all three discretizations provide comparable reconstruction results. 

\begin{figure}[H]
	\subfloat[]{%
		\includegraphics[width=0.24\textwidth]{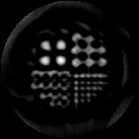}
		\label{subfig:3steps}} 
	\subfloat[]{%
		\includegraphics[width=0.24\textwidth]{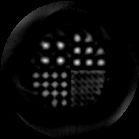}
		\label{subfig:4steps}} 
	\subfloat[]{%
	 	\includegraphics[width=0.24\textwidth]{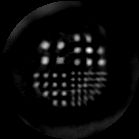}
	 	\label{subfig:8steps}} 
	\subfloat[]{%
		\includegraphics[width=0.24\textwidth]{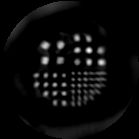}
		\label{subfig:30steps}} 
	\caption[]{Reconstruction with different number $N$ of nodes in the discretization of the Langevin function. The results given here are for 3, 4, 8 and 30 nodes. }
	\label{fig:FFL_steps_variation}
\end{figure}

\vspace{-6mm}

\begin{figure}[H]
	\input{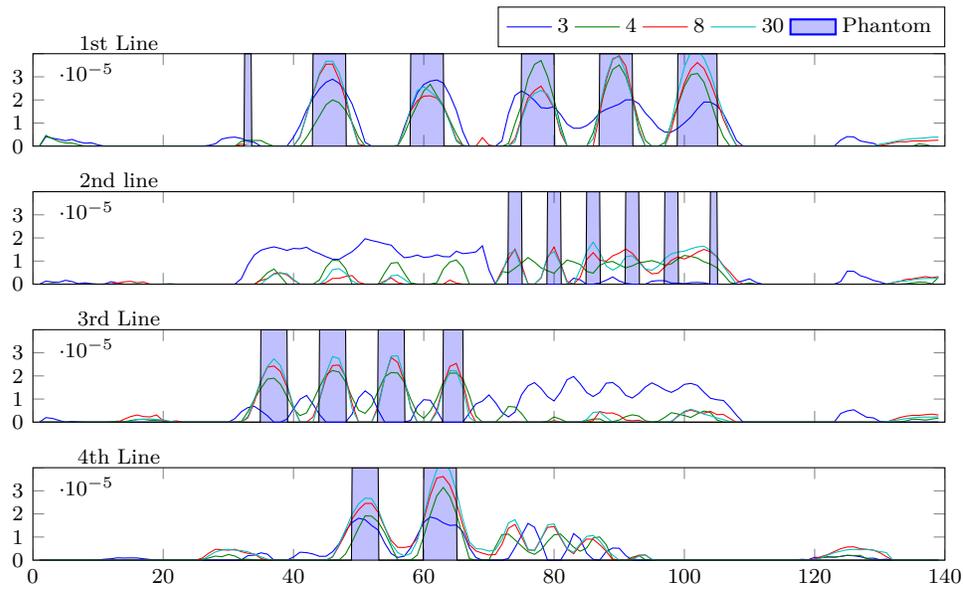}
	\caption[]{Line profiles (as shown in Fig.~\ref{subfig:phantom}) for the reconstructions obtained with our method using different numbers $N$ for the approximation of the Langevin function.}
	\label{fig:cut_steps}
\end{figure}

\vspace{-8mm}

\begin{figure}[H]
	\subfloat[]{%
		\includegraphics[width=0.32\textwidth]{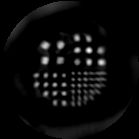}
		\label{subfig:thr1}} 
	\subfloat[]{%
		\includegraphics[width=0.32\textwidth]{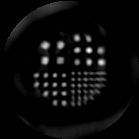}
		\label{subfig:thr2}} 
	\subfloat[]{%
	 	\includegraphics[width=0.32\textwidth]{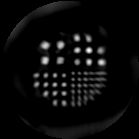}
	 	\label{subfig:thr3}} 
	\caption[]{Reconstruction with three different discretization schemes. \subref{subfig:thr1} Secant, equidistant; \subref{subfig:thr2} Tangent, equidistant; \subref{subfig:thr3} Tangent, $L_1$-optimal. }
	\label{fig:FFL_discre_variation}
\end{figure}

\begin{figure}[H]
	\input{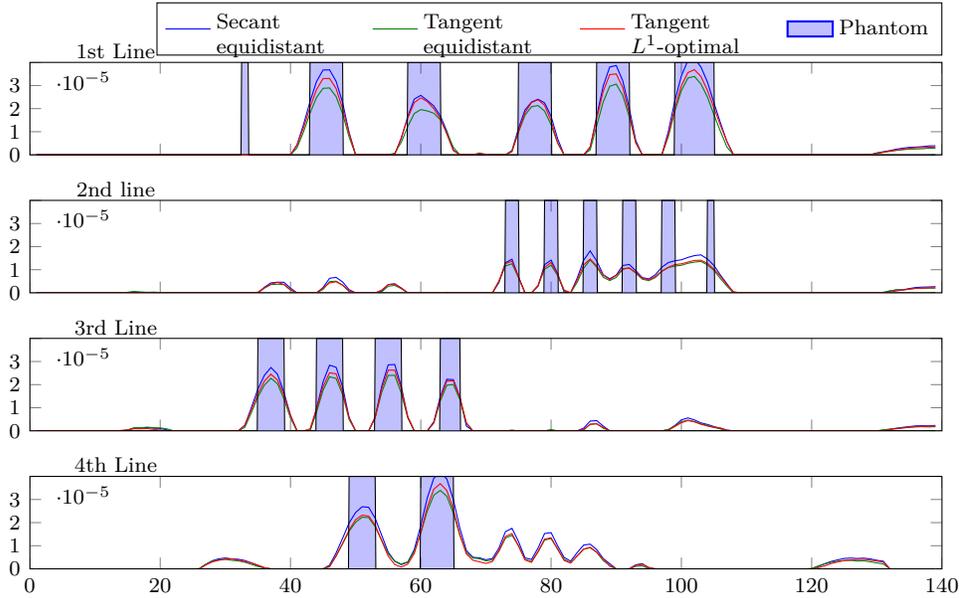}
	\caption[]{Line profiles (as shown in Fig.~\ref{subfig:phantom}) for the reconstructions obtained with our method using the three discretization techniques given in Fig. \ref{fig:FFL_discre_variation}.}
	\label{fig:cut_discret}
	\vspace{-2mm}
\end{figure}

\section{Conclusion} \label{sec:conclusion}

We introduced a new 3D modelling framework for magnetic particle imaging that includes classical MPI models based on ideal 1D-FFP and FFL magnetic fields, as well as realistic magnetic field topologies. Via expansions in spherical harmonics, this framework allows to incorporate realistic magnetic fields in the imaging model such that the reconstruction process can be adapted to a given scanner topology. In this sense, our framework can be regarded as a hybrid model-based approach for MPI in which the applied magnetic fields are measured in a preliminary calibration step and then included in the 3D model. 

Compared to an ideal FFP or FFL topology, the magnetic fields generated in real MPI scanners have distortions that lead to distorted low-field volumes. Our model-based approach is able to deal with these distortions and can generally be applied for magnetic fields that are parallel to their velocity field. 
We showed how this new 3D model can be approximated and discretized numerically in order to obtain a finite system matrix for the reconstruction of the magnetic particles. 

To obtain the final magnetic particle distribution, we inverted the resulting system matrix algebraically using a finite number of LSQR iterations and no further tuning. This was sufficient to evaluate the enhanced reconstruction properties of our proposed model, leaves however room for further improvements. In particular, the incorporation of more advanced regularization techniques is likely to have an additional positive effect on the reconstruction quality of our model.   

\newpage 
\section*{Acknowledgements}
The authors of this article were involved in the activities of the DFG-funded scientific network MathMPI (ER777/1-1) and thank the German Research Foundation for the support.

\begin{acronym}
	\acro{AS}[AS]{amplitude spectrum}
	\acro{ASD}[ASD]{amplitude spectral density}
	\acro{ART}[ART]{algebraic reconstruction technique}
	\acro{BEM}[BEM]{boundary element method}
	\acro{CS}[CS]{coordinate system}
	\acro{CT}[CT]{computed tomography}
	\acro{DCT}[DCT]{discrete cosine transform}
	\acro{DFT}[DFT]{discrete Fourier transform}
	\acro{FBP}[FBP]{filtered back-projection}
	\acro{FC}[FC]{Fourier coefficient}
	\acro{FEM}[FEM]{finite element method}
	\acro{FFL}[FFL]{field-free line}
	\acro{FFP}[FFP]{field-free point}
	\acro{FOV}[FOV]{field of view}
	\acro{HMIEFA}[HMIEFA]{highest maximal induced electrical field amplitude}
	\acro{ISIC}[ISIC]{idealised scanner made of idealised coils}
	\acro{LFV}[LFV]{low-field volume}
	\acro{LNA}[LNA]{low-noise amplifier}
	\acro{MC}[MC]{multipole coefficient}
	\acro{MIEFA}[MIEFA]{maximal induced electrical field amplitude}
	\acro{MPI}[MPI]{magnetic particle imaging}
	\acro{MRI}[MRI]{magnetic resonance imaging}
	\acro{MSE}[MSE]{multipole series expansion}
	\acro{nHMIEFA}[nHMIEFA]{normalised highest maximal induced electrical field amplitude}
	\acro{nMIEFA}[nMIEFA]{normalised maximal induced electrical field amplitude}
	\acro{PNS}[PNS]{peripheral nerve stimulation}
	\acro{PS}[PS]{power spectrum}
	\acro{PSD}[PSD]{power spectral density}
	\acro{QPQC}[QPQC]{quadratic programming problem with quadratic constrains}
	\acro{RMS}[RMS]{root mean square}
	\acro{SAR}[SAR]{specific absorption rate}
	\acro{SF}[SF]{system function}
	\acro{SHC}[SHC]{spherical harmonic coefficient}
	\acro{SPION}[SPION]{super\-para\-magnetic iron oxide nano\-particle}
    \acro{SPIONS}[SPIONs]{super\-para\-magnetic iron oxide nano\-particles}
	\acro{SNR}[SNR]{signal to noise ratio}
	\acro{SHSE}[SHSE]{spherical harmonic series expansion}
	\acro{THD}[THD]{total harmonic distortion}
	\acro{tSF}[tSF]{truncated system function}

\end{acronym}

\end{document}